\documentclass[12pt]{article}

\usepackage{amssymb}
\usepackage{amsmath}
\usepackage{amsfonts}
\usepackage{amsthm}
\usepackage{stmaryrd}
\usepackage[all]{xy}

\usepackage{fullpage}

\newtheorem{theorem}{Theorem}[subsection]
\numberwithin{equation}{theorem}

\newtheorem{corollary}[theorem]{Corollary}
\newtheorem{lemma}[theorem]{Lemma}
\newtheorem{prop}[theorem]{Proposition}

\theoremstyle{definition}
\newtheorem{caution}[theorem]{Caution}

\newtheorem{construction}[theorem]{Construction}
\newtheorem{convention}[theorem]{Convention}

\newtheorem{defn}[theorem]{Definition}

\newtheorem{example}[theorem]{Example}

\newtheorem{notation}[theorem]{Notation}

\newtheorem{remark}[theorem]{Remark}
\newtheorem{situation}[theorem]{Situation}
\newtheorem{hypo}[theorem]{Hypothesis}

\def\bbe{\mathbf{e}}

\def\bbv{\mathbf{v}}

\def\CC{\mathbb{C}}

\def\NN{\mathbb{N}}

\def\QQ{\mathbb{Q}}
\def\RR{\mathbb{R}}

\def\TT{\mathbb{T}}

\def\ZZ{\mathbb{Z}}

\def\calO{\mathcal{O}}

\def\calR{\mathcal{R}}

\def\gothm{\mathfrak{m}}

\def\gotho{\mathfrak{o}}

\def\gothI{\mathfrak{I}}

\def\gothR{\mathfrak{R}}

\newcommand{\serie}[2]{{#11},\dots,{#1{#2}}}
\newcommand{\seriezero}[2]{{#10}, \dots, {#1{#2}}}
\newcommand{\mat}[4]{\left(\begin{matrix}#1&#2\\#3&#4\end{matrix}\right)}
\newcommand{\exact}[5]{#1\rightarrow#2\rightarrow#3\rightarrow#4\rightarrow#5}

\newcommand{\rar}{\rightarrow}

\newcommand{\Rar}{\Rightarrow}

\newcommand{\LRar}{\Leftrightarrow}
\newcommand{\inj}{\hookrightarrow}
\newcommand{\surj}{\twoheadrightarrow}

\newcommand{\bs}{\backslash}

\newcommand{\GL}{\mathrm{GL}}

\newcommand{\Hom}{\mathrm{Hom}}

\newcommand{\rank}{\mathrm{rank}\,}

\newcommand{\Ker}{\mathrm{Ker}\,}
\newcommand{\Coker}{\mathrm{Coker}\,}

\newcommand{\dual}{\vee}

\DeclareMathOperator{\Frac}{Frac}

\newcommand{\alg}{\mathrm{alg}}

\newcommand{\disc}{\mathrm{disc}}

\newcommand{\cpltotimes}{\widehat{\otimes}}

\newcommand{\NP}{\mathrm{NP}}

\renewcommand{\sp}{\mathrm{sp}}

\newcommand{\gen}{\mathrm{gen}}

\DeclareMathOperator{\inte}{int}
\DeclareMathOperator{\Max}{Max}

\begin{document}

\title{Differential modules on $p$-adic polyannuli}

\author{Kiran S. Kedlaya \\ Department of Mathematics, Room 2-165 \\ 
Massachusetts
Institute of Technology \\ 77 Massachusetts Avenue \\
Cambridge, MA 02139 \\
\texttt{kedlaya@mit.edu} \\ \\
Liang Xiao \\ Department of Mathematics, Room 2-090 \\ Massachusetts Institute of Technology \\ 77 Massachusetts Avenue \\ Cambridge, MA 02139 \\
\texttt{lxiao@mit.edu}}

\date{December 15, 2008}

\maketitle

\begin{abstract}
We consider variational properties of some numerical invariants, measuring
convergence of local horizontal sections,
associated to differential modules on polyannuli over a nonarchimedean field
of characteristic zero.
This extends prior work in the one-dimensional case
of Christol, Dwork, Robba, Young, et al. Our results do not require
positive residue characteristic; thus besides their relevance to
the study of Swan conductors for isocrystals, 
they are germane to the formal classification
of flat meromorphic connections on complex manifolds.
\end{abstract}

\tableofcontents

\section*{Introduction}

Differential equations involving $p$-adic analytic functions have a nasty
habit of failing to admit global solutions even in the absence of
singularities; for instance, the exponential series fails to be entire.
To measure this, Dwork and his collaborators 
introduced the notion of
the \emph{generic radius of convergence} of a $p$-adic differential
module over a one-dimensional space (for simplicity, we restrict attention
here to discs and annuli). The modern definition of this
concept was given and studied in depth by Christol and Dwork 
\cite{christol-dwork}. A further refinement, the collection of
\emph{subsidiary generic radii of convergence}, was introduced (under
different terminology) by Young \cite{young}.

Given a differential module over a $p$-adic disc or annulus of the
form $\alpha \leq |t| \leq \beta$, one obtains
a generic radius of convergence and some subsidiary radii for each
radius $\rho \in [\alpha, \beta]$, and one would like to be able to say 
something about how these quantities vary with $\rho$. 
(In fact, one also obtains these data for each point of the Berkovich1
analytic space; this is the point of view adopted in
ongoing work of Baldassarri and di Vizio, starting with
\cite{baldassarri-divizio}.) By pulling together techniques from the literature
and adding one or two new ideas, one can make fairly definitive statements
about the nature of this variation; this was done by the first author
in a course given in fall 2007, whose compiled notes constitute
the volume \cite{kedlaya-course}.

The course \cite{kedlaya-course} was deliberately restricted to
the study of $p$-adic \emph{ordinary} differential equations. One could view
the extension of the variational results to higher-dimensional spaces
as an implied exercise in \cite{kedlaya-course}. This paper constitutes
a partial solution of this implied exercise, in which we
obtain variational properties for differential modules over certain
higher-dimensional $p$-adic analytic spaces. The spaces we consider
are what one might call \emph{generalized polyannuli}: such a space
is an analytic subspace of an affine space in some variables $t_1,
\dots, t_n$, defined by the restriction $(|t_1|, \dots, |t_n|) \in S$
for some set $S$ such that $\log S$ is convex. (In order for this to actually
define an analytic space, one must impose some polyhedrality conditions on
$\log S$. For an example of what happens when such conditions are missing,
see the treatment of ``fake annuli'' in \cite{kedlaya-fake}.)

The strategy we adopt is to proceed in three stages. 
We start with some formalism for differential modules over differential
fields (corresponding to zero-dimensional spaces), in somewhat greater generality
than in \cite{kedlaya-course}. We then make a series of calculations
on a one-dimensional
annulus over a nonarchimedean field which itself carries one or more
commuting derivations. We consider modules equipped with commuting actions of 
both the base derivations and the derivation in the geometric direction,
and obtain results in the spirit of those in \cite{kedlaya-course}.
We finally extend these results
to higher-dimensional spaces (which may still
carry derivations on the base field) by using some careful analysis of
convex functions on polyhedral subsets of $\RR^n$.

The original 
intended application of these results is to the study of differential
Swan conductors for isocrystals, as introduced by the first
author in \cite{kedlaya-swan1}.
(The extra work in Section~\ref{sec:over field} is needed to obtain a
common generalization of the hypotheses in \cite{kedlaya-course} and
\cite{kedlaya-swan1}.)
The deployment of these results in the study of differential Swan
conductors takes place in \cite{kedlaya-swan2},
following up on earlier investigations by
Matsuda \cite{matsuda-dwork}.
Since our results do not require positive residual characteristic,
they are also relevant to formal classification of flat meromorphic
connections on complex manifolds, as in the work of 
Sabbah \cite{sabbah} for complex analytic surfaces.
For instance, the first author
\cite{kedlaya-goodformal1} recently used the results of this paper to
resolve the main conjecture of \cite{sabbah}.

\subsection*{Acknowledgments}
Both authors were supported by 
NSF CAREER grant DMS-0545904. 
The first author was also supported by 
a Sloan Research Fellowship.

\section{Differential modules over a field}
\label{sec:over field}

In this section, we assemble  a slightly more comprehensive collection of definitions
and basic results  concerning differential modules over a field
than was given in \cite[\S 1]{kedlaya-part1}. This is done
in order to state results applicable in the context of \cite{kedlaya-swan1}.

\subsection{Setup}
\label{S:setup-1}

\begin{convention}
Let $f^*: R_1 \rar R_2$ be a homomorphism of rings.  For an $R_1$-module $M_1$, we write $f^*M_1$ to denote the extension of scalars $M_1 \otimes_{R_1, f^*} R_2$.  For an $R_2$-module $M_2$ we write $f_* M_2$ to mean $M_2$ viewed as an $R_1$-module via $f^*$ (i.e., the restriction of scalars).
\end{convention}

\begin{convention}
For any nonarchimedean field $K$ of characteristic zero, denote its ring of integers and residue field by $\gotho_K$ and $k$, respectively.  
We reserve the letter $p$ for the residual characteristic of $K$.
If $p>0$,
we normalize the norm $|\cdot|$ on $K$ so that $|p| = 1/p$.  For an element $a \in \gotho_K$, we denote its reduction in $k$ by $\bar a$.  In case $K$ is discretely valued, let $\pi_K$ denote a uniformizer of $\gotho_K$.
\end{convention}

\begin{defn}
A finite extension $L$ of a complete nonarchimedean field $K$ is \emph{unramified} if $L$ and $K$ have the same value group, and the residue field extension is separable of degree $[L:K]$.  It is \emph{tamely ramified} if the index $e_{L/K}$ of the value group of $K$ in that of $L$ is not divisible by $p$,
and the residue field extension is separable of degree $[L:K]/e_{L/K}$;
we call $e_{L/K}$ the \emph{ramification degree}.  
If $p=0$, then any finite extension of $K$ is tamely ramified,
by a theorem of Ostrowski (see \cite[Chapter~6]{ribenboim}).
For $L$ the completion of an infinite algebraic extension of $K$, 
we say that $L$ is unramified or tamely ramified if the same is true of
each finite subextension of $L$ over $K$; we define the ramification degree
to be the supremum of the ramification degrees of the finite subextensions.
\end{defn}

\begin{convention}
Let $J$ be a finite index set.  We will write $e_J$ for a tuple $(e_j)_{j \in J}$.  For another tuple $u_J$, write $u_J^{e_J} = \prod_{j \in J} u_j^{e_j}$. We also use $\sum_{e_J = 0}^{n}$ to mean the sum over $e_j \in \{0, 1, \dots, n\}$ for each $j \in J$; for notational simplicity, we may suppress the range of the summation when it is clear.  Write $|e_J| = \sum_{j \in J}|e_j|$ and $(e_J)!$ for $\prod_{j \in J} (e_j)!$.
\end{convention}

\begin{convention}
For a matrix $A = (A_{ij})$ with coefficients in a nonarchimedean ring, we use $|A|$ to denote the supremum norm over entries.
\end{convention}

\begin{hypo}
For the rest of this subsection, we assume that $K$ is a complete nonarchimedean field.
\end{hypo}

\begin{notation}\label{N:affinoids}
Let $I \subset [0, +\infty)$ be an interval. Let $A_K^1(I)$ denote
the annulus with radii in $I$.
(We do not impose any rationality condition on the endpoints of $I$, so this
space should be viewed as an analytic space in the sense of 
Berkovich \cite{berkovich}.)
If $I$ is written explicitly in terms of its
endpoints (e.g., $[\alpha, \beta]$),
we suppress the parentheses around $I$ (e.g., 
$A_K^1[\alpha, \beta]$).
For $0 \leq \alpha \leq \beta < \infty$, let 
$K \langle \alpha/t, t/\beta \rangle$ denote the ring of analytic functions
on $A_K^1[\alpha, \beta]$. (If $\alpha = 0$, we write $K \langle t/\beta \rangle$
instead.)
\end{notation}

\begin{defn}
We have the ring of \emph{series with bounded coefficients}
\[
K \llbracket t/\beta \rrbracket_0 
= \left\{ \sum_{i=0}^\infty a_i t^i \in K \llbracket t \rrbracket: \sup_i \{|a_i| \beta^i\} < \infty\right\};
\]
these are the power series which converge and take bounded values on the open disc $|t| < \beta$.
Note that for any $\delta \in (0,\beta)$,
\[
K \langle t/\beta \rangle \subset
K \llbracket t/\beta \rrbracket_0 
\subset K \langle t/\delta \rangle.
\]
In particular, when $\beta = 1$, we have
\[
K \llbracket t \rrbracket_0 = \gotho_K \llbracket t \rrbracket \otimes_{\gotho_K} K.
\]

An analogue of this construction for an annulus is
\[
K \langle \alpha/t, t/\beta \rrbracket_0
= \left\{ \sum_{i \in \ZZ}
a_i t^i: a_i \in K, 
\lim_{i \to - \infty} |a_i| \alpha^i = 0,
\sup_i \{|a_i| \beta^i \} < \infty \right\};
\]
these are the Laurent series which converge and take bounded values on the half-open annulus $\alpha \leq |t| < \beta$.
For any $\delta \in [\alpha, \beta)$, this ring satisfies
\[
K \langle \alpha/t, t/\beta \rangle \subset
K \langle \alpha/t, t/\beta \rrbracket_0 
\subset K \langle \alpha/t, t/\delta \rangle.
\]
\end{defn}

\begin{defn}
Define the ring
$$
K \{\{ t/\beta \}\} = \bigcap_{\delta \in (0,\beta)} K \langle t/\delta \rangle = \left\{ \sum_{i=0}^\infty a_i t^i: a_i \in K, \lim_{i \to \infty} |a_i| \rho^i =0 \quad(\rho \in (0,\beta)) \right\};
$$
these are the power series convergent on the open disc $|t| < \beta$,
with no boundedness restriction. In particular, for any $\delta \in (0, \beta)$,
\[
K \llbracket t/\beta \rrbracket_0 \subset
K \{\{t/\beta\}\} \subset K \langle t/\delta \rangle.
\]

An analogue of the previous construction for an annulus is
$$
K \{\{ \alpha/t, t/\beta \}\} = \left\{ \sum_{i \in \ZZ} a_i t^i: a_i \in K, 
\lim_{i \to \pm \infty} |a_i| \eta^i = 0 \quad(\eta \in (\alpha, \beta)) \right\};
$$
these are the Laurent series convergent on the open annulus $\alpha < |t| < \beta$.
\end{defn}

\begin{defn}
Put $I = \{\serie{}n\}$.  For $(\eta_i)_{i \in I} \in (0, +\infty)^n$, the \emph{$\eta_I$-Gauss norm} on $K[t_I]$ is the norm $|\cdot|_{\eta_I}$ given by
$$
\left| \sum_{e_I} a_{e_I} t_I^{e_I}\right|_{\eta_I} = \max \left\{ |a_{e_I}| \cdot \eta_I^{e_I}\right\};
$$
this norm extends uniquely to $K(t_I)$.

For $\eta \in [\alpha, \beta]$ and $\eta \neq 0$, let $x = \sum_{i = -\infty}^\infty a_i t^i$
be an element of $K \langle \alpha / t, t / \beta \rangle$, $K \langle \alpha / t, t / \beta \rrbracket_0$, or (if $\eta \neq \alpha,\beta$)
$K \{\{ \alpha / t, t / \beta \}\}$.  
We define the \emph{$\eta$-Gauss norm} of $x$ to be 
$$
|x|_\eta = \sup \left\{ |a_i| \cdot \eta^i \right\}.
$$
\end{defn}

\begin{convention}\label{N:G-maps}
By a \emph{$G$-map}, we will mean a 
morphism of affinoid or Stein ($K$-)analytic spaces with $G$-topology, which 
need not respect the $K$-space structure. 
This amounts to a homomorphism between the corresponding rings
of global sections, which need not be $K$-linear.
 For example, the homomorphism $f_\gen^*$ defined in Lemma~\ref{L:generic-point}
below gives rise to a $G$-map $f_\gen: A_K^1[0,R_\partial(K)) \rar \Max(K)$.
\end{convention}

\begin{convention}
Throughout this paper, all derivations on topological modules will
be assumed to be continuous; moreover, any derivation considered on a ring
equipped with a nonarchimedean norm will be assumed to be bounded (i.e.,
to have bounded operator norm). All connections considered will be
assumed to be integrable. 
We may suppress the base ring from a module of continuous differentials when it is 
unambiguous.
\end{convention}

\subsection{Differential fields and differential modules}
\label{S:dif-field}

\begin{defn}\label{D:twisted-poly}
Let $K$ be a differential ring of order 1, i.e., a ring equipped with a derivation $\partial$.  Let $K\{T\}$ denote the (noncommutative) ring of twisted polynomials over $K$ \cite{ore}; its elements are finite formal sums $\sum_{i \geq 0} a_iT^i$ with $a_i \in K$, multiplied according to the rule $Ta = aT + \partial(a)$ for $a \in K$.
\end{defn}

\begin{defn}
A \emph{$\partial$-differential module} over $K$ is a finite projective $K$-module $V$ equipped with an action of $\partial$ (subject to the Leibniz rule); any $\partial$-differential module over $K$ inherits a left action of $K\{T\}$ where $T$ acts via $\partial$.  The module dual $V^\dual  = \Hom_K(V,K)$ of $V$
may be viewed as a $\partial$-differential module
by setting $(\partial f)(\bbv) = \partial(f(\bbv)) - f(\partial(\bbv))$.
We say $V$ is \emph{free} if $V$ as a module is free over $K$.  We say $V$ is \emph{trivial} if it is free and there exists a $K$-basis $\serie{\bbv_}d \in V$ such that $\partial(\bbv_i) = 0$ for $i = \serie{}d$.

For $V$ a differential module over $K$, we say $\bbv \in V$ is a \emph{cyclic vector} if $\bbv, \partial \bbv, \dots, \partial^{\rank(V) - 1}\bbv$ form a basis of $V$.  A cyclic vector defines
an isomorphism $V \simeq K\{T\} / K \{T\}P$ of differential modules for some twisted polynomial
$P \in K\{T\}$, where the $\partial$-action on $K\{T\}/K\{T\}P$ is the left multiplication by $T$.
\end{defn}

\begin{defn}
For a differential module $V$ over $K$, define
$$
H^0_\partial(V) = \Ker \partial, \quad H^1_\partial(V) = \Coker \partial = V / \partial (V).
$$
The latter computes Yoneda extensions; see, e.g.,
\cite[Lemma~5.3.3]{kedlaya-course}.
\end{defn}

\begin{lemma} \label{L:cyclic vector}
If $K$ is a field of characteristic zero, every differential module over $K$ contains a cyclic vector.
\end{lemma}
\begin{proof}
See, e.g., \cite[Theorem~III.4.2]{dgs} or \cite[Theorem~5.4.2]{kedlaya-course}.
\end{proof}

\begin{hypo}
For the rest of this subsection,
we assume that $K$ is a field of characteristic zero complete for a
nonarchimedean norm $|\cdot|$ and equipped with a derivation $\partial$ with
operator norm $|\partial|_K < \infty$,
and that $V$ is a nonzero $\partial$-differential module over $K$.
\end{hypo}

\begin{defn}\label{D:spectral-norms}
The \textit{spectral norm of $\partial$ on $V$} is defined to be
$$
|\partial|_{\sp, V} = \lim_{n \rar \infty} |\partial^n|_V^{1/n}
$$
for any fixed $K$-compatible norm $|\cdot|_V$
on $V$.  Any two such norms on $V$ are equivalent \cite[Proposition~4.13]{schneider}, so the spectral norm
does not depend on the choice \cite[Proposition~6.1.5]{kedlaya-course}.  One can show that $|\partial|_{\sp, V} \geq |\partial|_{\sp, K}$ \cite[Lemma~6.2.4]{kedlaya-course}.

Explicitly, if one chooses a basis of $V$ and lets $D_n$ denote the matrix via which $\partial^n$ acts on this basis, then
$$
|\partial|_{\sp, V} = \max\{|\partial|_{\sp, K}, \lim_{n \rar \infty} |D_n|^{1/n} \}.
$$
\end{defn}

\begin{remark}\label{R:base-change-sp-norm}
If $K \rar K'$ is an isometric embedding of complete nonarchimedean differential fields,
then 
for a differential module $V$ over $K$, $V' = V \otimes_K K'$ is a differential module over $K'$, and $|\partial|_{\sp, V'} = \max \left\{ |\partial|_{\sp, K'}, |\partial|_{\sp, V} \right\}$.
\end{remark}

\begin{defn}\label{D:cvgt-radii}
Let $p$ denote the residual characteristic of $K$; we conventionally write
\[
\omega = \begin{cases} 1 & p = 0 \\ p^{-1/(p-1)} & p > 0 \end{cases}.
\]
Define the \emph{generic $\partial$-radius of convergence} (or for short, the \emph{generic $\partial$-radius}) of $V$ to be
$$
R_\partial(V) = \omega|\partial|_{\sp, V}^{-1};
$$
note that $R_\partial(V) > 0$.  We will see later (Proposition~\ref{P:generic-point}) that this indeed computes the radius of convergence of Taylor series on a ``generic disc".  In some situations, it is more natural to consider the \emph{intrinsic generic $\partial$-radius of convergence}, or for short the \emph{intrinsic $\partial$-radius}, defined as
$$
IR_\partial(V) = \frac {|\partial|_{\sp, K}} {|\partial|_{\sp, V}};
$$
note that this is a number in $(0,1]$.

Let $\serie{V_}d$ be the Jordan-H\"older constituents of $V$.  We define the \emph{(extrinsic) subsidiary generic $\partial$-radii of convergence}, or for short the \emph{subsidiary $\partial$-radii}, to be the multiset $\gothR_\partial(V)$ consisting of $R_\partial(V_i)$ with multiplicity $\dim V_i$ for $i = \serie{}d$.  Let $R_\partial(V;1) \leq \cdots \leq R_\partial(V;\dim V)$ denote the elements in $\gothR_\partial(V)$ in increasing order.
We similarly define \emph{intrinsic subsidiary (generic) $\partial$-radii of convergence} $\gothI\gothR_\partial(V)$, or for short \emph{intrinsic subsidiary $\partial$-radii}, by aggregating the intrinsic $\partial$-radii of $V_i$ for $i = \serie{}d$.  Let $IR_\partial(V;1) \leq \cdots \leq IR_\partial(V;\dim V)$ denote the elements in $\gothI\gothR_\partial(V)$ in increasing order.

We say that $V$ has \emph{pure $\partial$-radii} if $\gothR(V)$ consists of $d$ copies of $R_\partial(V)$.
\end{defn}

\begin{lemma}\label{L:basic-IR-prop}
Let $V$, $V_1$, $V_2$ be nonzero $\partial$-differential modules over $K$.

(a) For $\exact 0 {V_1} V {V_2} 0$ exact,
$$
R_\partial(V) = \min \left\{ R_\partial(V_1),\ R_\partial(V_2) \right\}; \quad IR_\partial(V) = \min \left\{ IR_\partial(V_1),\ IR_\partial(V_2) \right\}.
$$
More precisely,
$$
\gothR_\partial(V) = \gothR_\partial(V_1) \cup \gothR_\partial(V_2); \quad \gothI\gothR_\partial(V) = \gothI\gothR_\partial(V_1) \cup \gothI\gothR_\partial(V_2).
$$

(b) We have
\begin{eqnarray*}
R_\partial(V^\dual) = R_\partial(V ); && IR_\partial(V^\dual) = IR_\partial(V ); \\
\gothR_\partial(V^\dual) = \gothR_\partial(V ); && \gothI\gothR_\partial(V^\dual) = \gothI\gothR_\partial(V );
\end{eqnarray*}

(c) We have
$$
R_\partial(V_1 \otimes V_2) \geq \min \left\{ R_\partial(V_1),\ R_\partial(V_2) \right\}; \quad IR_\partial(V_1 \otimes V_2) \geq \min \left\{ IR_\partial(V_1),\ IR_\partial(V_2)\right\},
$$
with equality when $R_\partial(V_1) \neq R_\partial(V_2)$, or equivalently, 
when $IR_\partial(V_1) \neq IR_\partial(V_2)$.

(d) If $V_1$ and $V_2$ are irreducible and $IR_\partial(V_1) \neq IR_\partial(V_2)$, then $\gothI\gothR_\partial(V_1 \otimes V_2)$ is just $\dim V_1 \cdot \dim V_2$ copies of $\min\{IR_\partial(V_1), IR_\partial(V_2)\}$
\end{lemma}
\begin{proof}
As in \cite[Lemma~6.2.8]{kedlaya-course} and \cite[Corollary~6.2.9]{kedlaya-course}.
\end{proof}

\begin{defn}\label{D:Taylor-series}
Let $R$ be a complete $K$-algebra.  For $\bbv \in V$ and $x \in R$, define the \emph{$\partial$-Taylor series} to be
$$
\TT (\bbv; \partial, x) = \sum_{n = 0}^\infty \frac {\partial^n (\bbv)}{n!}x^n \in V \cpltotimes_K R
$$
in case this series converges.
\end{defn}

\begin{remark}
If $V=K$, the $\partial$-Taylor series gives a ring homomorphism $K \rar R$ if it converges.  For general $V$, the $\partial$-Taylor series gives a homomorphism of modules $V \rar V \cpltotimes R$ via the aforementioned ring homomorphism, if it converges.
\end{remark}

\begin{lemma}\label{L:generic-point}
The Taylor series $x \mapsto \TT(x; \partial, T)$ gives a continuous 
homomorphism $f_\gen^*: K \rar K \llbracket T / R_\partial(K) \rrbracket_0$, 
which induces a $G$-map $f_\gen: A_K^1[0, R_\partial(K)) \rar \Max(K)$.  
Moreover, for $\eta \in [0, R_\partial(K)]$, $f^*_\gen$ is isometric for the
$\eta$-Gauss norm on the target.
\end{lemma}
\begin{proof}
It is straightforward to check that $f_\gen^*$ is bounded for the $\eta$-Gauss
norm for any $\eta \in [0, R_\partial(K))$; that is, there exists
$c>0$ such that for all $x \in K$, $|f^*_\gen(x)|_\eta \leq c|x|$. For any
positive integer $n$, we can plug $x^n$ into the previous inequality to deduce
$|f^*_\gen(x)|_\eta \leq c^{1/n}|x|$. Consequently,
$|f^*_\gen(x)|_\eta \leq |x|$ for any $\eta \in [0, R_\partial(K))$, and by
continuity also for $\eta = R_\partial(K)$.
\end{proof}

\begin{corollary} \label{C:bound using spectral norm}
For each positive integer $n$,
we have $|\partial^n/n!|_K \leq R_\partial(K)^{-n} = \omega^{-n}
|\partial|^n_{\sp,K}$. In particular (by taking $n=1$), 
$|\partial|_{\sp,K} \geq \omega
|\partial|_K$.
\end{corollary}

We have the following geometric interpretation of generic radii.
This is slightly different from, but essentially equivalent to,
the treatments in 
\cite[Section~2.2]{kedlaya-swan1} and \cite[Section~9.7]{kedlaya-course}.

\begin{prop}\label{P:generic-point}
With notation as in Lemma~\ref{L:generic-point},
the pullback $f_\gen^*V$ becomes a $\partial_T$-differential module 
over $A_K^1[0, R_\partial(K))$, where $\partial_T = \frac d{dT}$.  Then for any $r \in (0, R_\partial(K)]$, $R_\partial(V) \geq r$ if and only if $f_\gen^*V$ restricts to a trivial $\partial_T$-differential module over $A_K^1[0, r)$.
\end{prop}
\begin{proof}
Since $f_\gen^*$ is an isometry and 
$|\partial_T|_{K \llbracket T/R_{\partial}(K) \rrbracket_0} = 
R_{\partial}(K)^{-1}$, we have $R_\partial(V) = R_{\partial_T}(f_\gen^* V
\otimes \Frac K \llbracket T/R_{\partial}(K) \rrbracket_0)$.
It then suffices to check that $R_{\partial_T}(f_\gen^* V) \geq r$ if and only
if $f_{\gen}^* V$ restricts to a trivial $\partial_T$-differential module
over $A^1_K[0,r)$; this is the content of Dwork's transfer theorem
\cite[Theorem~9.6.1]{kedlaya-course}.
\end{proof}

\subsection{Newton polygons}
\label{S:NP}

In this subsection, we summarize some results in 
\cite[Chapter~5~and~6]{kedlaya-course} and \cite[Section~1]{kedlaya-swan1}.
Throughout this subsection, let $K$ be a complete nonarchimedean differential field of
characteristic zero.

\begin{defn}
For $P(T) = \sum_i a_iT^i \in K\{T\}$ a nonzero twisted polynomial, define the \emph{Newton polygon} of $P$ as the lower convex hull of the set $\{(-i, -\log|a_i|)\} \subset \RR^2$. This Newton polygon obeys the usual additivity rules only for slopes less than $-\log |\partial|_K$.
\end{defn}

\begin{prop}[Christol-Dwork]
\label{P:spec-norm-from-NP}
Suppose $V \simeq K\{T\} / K\{T\}P$,
and let $s$ be the lesser of $-\log |\partial|_K$ and the
least slope of $P$. Then
$$
\max\{|\partial|_K, |\partial|_{\sp, V} \} = e^{-s}.
$$
\end{prop}
\begin{proof}
See \cite[Th\'eor\`eme~1.5]{christol-dwork} or \cite[Theorem~6.5.3]{kedlaya-course}.
\end{proof}

\begin{prop}[Robba]\label{P:NP-decomp-field}
Any monic twisted polynomial $P\in K\{T\}$ admits a unique factorization
$$
P = P_+ P_n \cdots P_1
$$
such that for some $s_1 < \cdots < s_n < -\log |\partial|_K$, each $P_i$ is monic with all slopes equal to $s_i$, and $P_+$ is monic with all slopes at least $-\log |\partial|_K$.
\end{prop}
\begin{proof}
See \cite[Proposition~1.1.10]{kedlaya-swan1} or \cite[Corollary~3.2.4]{kedlaya-part3}.
\end{proof}

\begin{prop}\label{P:part-decomp-over-field}
Suppose that $\omega \cdot |\partial|_K^{-1} = r_0$.  Then there is a unique decomposition 
$$
V = V_+ \oplus \bigoplus_{r < r_0} V_r
$$
of differential modules, such that $V_r$ has pure $\partial$-radii $r$, and 
the subsidiary radii of $V_+$ are all at least $r_0$.
\end{prop}
\begin{proof}
Apply Lemma~\ref{L:cyclic vector}
to write $V \simeq K\{T\} / K\{T\}P$ for $P$ a twisted polynomial.  Then the statement may be deduced
from Proposition~\ref{P:NP-decomp-field}, applied first to
$P$ in $K\{T\}$ and then to $P$ in the opposite ring.  
For more details, one may consult \cite[Theorem~6.6.1]{kedlaya-course}.
\end{proof}

\begin{remark}\label{R:part-decomp-over-field}
If $V \simeq K\{T\} / K\{T\}P$ for $P$ a twisted polynomial, then 
Propositions~\ref{P:spec-norm-from-NP} and~\ref{P:NP-decomp-field} imply that the 
multiplicity of any $s < -\log |\partial|_K$ as a slope of the Newton polygon 
of $P$ coincides with the multiplicity of $\omega e^s$ in $\calR_\partial(V)$.
\end{remark}

\subsection{Moving along Frobenius}
\label{S:Frob}

As discovered originally by Christol-Dwork \cite{christol-dwork},
and amplified by the first author \cite{kedlaya-course}, 
in the situation of Definition~\ref{D:admissible-operator}, one can
overcome the limitation on subsidiary radii imposed by 
Proposition~\ref{P:spec-norm-from-NP} by using the pushforward 
along the Frobenius.
In this subsection, we imitate the techniques in 
\cite[Chapter~10]{kedlaya-course} 
and obtain Theorems~\ref{T:Frob-push-sp-norm} and 
\ref{T:decomp-over-field-complete}
as analogues of \cite[Theorems~10.5.1~and~10.6.2]{kedlaya-course}.

\begin{defn}\label{D:admissible-operator}
Let $K$ 
be a complete nonarchimedean differential field of characteristic zero
and residual characteristic $p$.
The derivation $\partial$ on $K$ is of \emph{rational type} if there exists
$u \in K$ such that the following conditions hold. (If these hold,
we call $u$ a \emph{rational parameter} for $\partial$.)
\begin{enumerate}
\item[(a)]
We have $\partial(u) = 1$ and $|\partial|_K = |u|^{-1}$.
\item[(b)]
For each positive integer $n$, $|\partial^n/n!|_K \leq |\partial|_K^n$.
\end{enumerate}
It is equivalent to formulate (b) as follows. 
\begin{enumerate}
\item[(b$'$)]
We have $|\partial|_{\sp,K} \leq \omega |\partial|_K$.
\end{enumerate}
(It is clear that (b) implies
(b$'$); the reverse implication holds by 
Corollary~\ref{C:bound using spectral norm}.)
For $p>0$, in 
the presence of (a), yet another equivalent formulation of (b) is as follows.
\begin{enumerate}
\item[(b$''$)]
For each polynomial $P \in \QQ_p[T]$ such that $P(\ZZ_p) \subseteq \ZZ_p$,
$|P(u \partial)|_K \leq 1$.
\end{enumerate}
This relies on the fact that the $\ZZ_p$-module of such $P$ is freely generated
by the binomial polynomials
\[
\binom{T}{n} = \frac{T(T-1)\cdots(T-n+1)}{n!} \qquad (n=0,1,\dots).
\]
\end{defn}

\begin{remark}
Note that in Definition~\ref{D:admissible-operator}, the inequality in 
(b$'$) is forced to be an equality by Corollary~\ref{C:bound using spectral norm},
while the inequality in (b) is forced to be 
an equality if (a) holds because then $(\partial^n/n!)(u^n) = 1$.
In particular, for any nonzero $\partial$-differential module $V$,
$IR_\partial(V) = |u| \cdot R_\partial(V)$.
Similarly, if (a) holds and $p>0$, then the inequality in (b$''$) becomes
an equality whenever $P(\ZZ_p)\not\subset p\ZZ_p$.
\end{remark}

\begin{remark}
If $u'$ is a second rational parameter for $\partial$, then
$u - u' \in \ker(\partial)$ and $|u-u'| \leq |u|$. The converse is also true;
that is, if $u$ is a rational parameter,
$u - u' \in \ker(\partial)$, and $|u-u'| \leq |u|$, then
$u'$ is also a rational parameter.
The only nonobvious part of this statement is the fact that these two
conditions imply $|u'| = |u|$. It is clear that $|u'| \leq |u|$; on the other hand,
since $\partial(u') = 1$, 
$1 \leq |\partial|_K |u'| = |u'|/|u|$, so $|u'| \geq |u|$.
\end{remark}

\begin{remark}
The simplest case of Definition~\ref{D:admissible-operator} is the derivation
$d/dt$ on the completion of the rational function field $\QQ_p(t)$ for any
Gauss norm if $p>0$, or on the ring of Laurent series $\CC((t))$ if $p=0$. For more cases, see 
Situation~\ref{Sit:K-in-SwanI} and the following
remarks.
\end{remark}

\begin{lemma}\label{L:H-under-unram-extn}
Let $L/K$ be a complete tamely ramified extension of $K$. Then the unique extension of
$\partial$ to $L$ is of rational type (with $u$ again as rational parameter).
\end{lemma}
\begin{proof}
We reduce immediately to the case of a finite tamely ramified extension.
The extension of $\partial$ to $L$ is obtained from the isomorphism
$\Omega^1_L \cong L \otimes \Omega^1_K$.
We need to prove that for each positive integer $n$ and each 
$x \in L$,
$|u^n \partial^n(x) / n!| \leq |x|$.  
We may consider the unramified extension and the totally tamely ramified extension separately.

Suppose first that $L/K$ is unramified. 
Since every element of $L$ equals an element of $K$ times an element of
$\gotho^\times_L$, we need only check the inequality
$|u^n \partial^n(x) / n!| \leq |x|$ for $x \in \gotho^\times_L$.
We do this by induction on $n$.
Let $h(T) = T^d + a_{d-1} T^{d-1} + \cdots + a_0 \in \gotho_K[T]$
be the minimal polynomial of $x$; thus $h'(x) \in \gotho_L^\times$. 
For the base case $n = 1$ of the induction, applying $u\partial$ to 
the equation $h(x) = 0$ gives
$$
u\partial(x) = -\frac {u\partial(a_{d-1}) x^{d-1} + \cdots + u\partial(a_0)}{h'(x)} \in \gotho_L.
$$
Assume the statement is proved for $n-1$.  Applying $u^n \partial^n / n!$ to 
the equation $h(x) = 0$ gives
$$
\sum_{i = 0}^d \sum_{\lambda_0 + \cdots + \lambda_i = n} \frac{u^{\lambda_0}\partial^{\lambda_0}} {\lambda_0!}(a_i) \frac{u^{\lambda_1}\partial^{\lambda_1}} {\lambda_1!}(x)\cdots \frac{u^{\lambda_i}\partial^{\lambda_i}} {\lambda_i!}(x) = 0,
$$
where $a_d = 1$ by convention.
Each summand belongs to $\gotho_L$ by the induction hypothesis except for those
in which $\lambda_j = n$ for some $j > 0$; those terms add up to
$h'(x) u^n \partial^n(x)/n!$. Therefore 
$u^n\partial^n(x) / n! \in \gotho_L$, completing the induction.

Now suppose that $L/K$ is totally tamely ramified. We induct on
$[L:K]$, which we may assume is greater than 1.
Then we can find $d>1$ and $x_0 \in \gotho_L$ 
such that $|x_0^i| \notin |K^\times|$ 
for $i = 1, \dots, d-1$.  Choose an element
$y \in \gotho_K$ with $|y-x_0^d| < |x_0^d|$. By Hensel's lemma, 
$y$ has a $d$-th root $z$ in $L$.
Let $K'$ be the completion of $K(t)$ for the $|y|^{1/d}$-Gauss norm,
and extend $\partial$ to $K'$ by setting $\partial(t) = 0$.
The residue field of $K'$ is $k(y/t^d)$.
Put $L' = K' \otimes_K K(z)$; then $L' = K'(z) = K'(z/t)$.
Now $z/t$ is a $d$-th root of the quantity $y/t^d \in \gotho_{K'}$,
whose image in the residue field has no $i$-th root for any
$i > 1$ dividing $d$. Hence $L'/K'$ is unramified, so
by the previous paragraph,
$\partial$ extends to $L'$ and is of rational type with respect to $u$.
We may then read off the same conclusion for $K(z)$; applying the
induction hypothesis to $L/K(z)$ yields the claim.
\end{proof}

\begin{hypo}
For the rest of this subsection, we assume that $K$ is a complete nonarchimedean field 
of characteristic zero and residual characteristic $p$, 
equipped with a differential operator 
$\partial$ of rational type with respect to the rational parameter $u$.
We also assume $p>0$ unless otherwise specified.
\end{hypo}

\begin{construction}\label{Cstr:j-frobenius}
If $K$ contains a primitive $p$-th root of unity $\zeta_p$, we may define an action of the group $\ZZ/p\ZZ$ on $K$ using $\partial$-Taylor series:
$$
x^{(i)} = \TT(x; \partial, (\zeta_p^i - 1) u), \qquad (i \in \ZZ / p\ZZ, x \in K).
$$
It is clear that $|x^{(i)}| = |x|$ for $i \in \ZZ / p\ZZ$.  Let $K^{(\partial)}$ be the fixed subfield of $K$ under this action; in particular, $u^p \in K^{(\partial)}$.  By simple Galois theory, $K$ is a Galois extension of $K^{(\partial)}$ generated by $u$ with Galois group $\ZZ / p\ZZ$.  Moreover, $K^{(\partial)}$ is stable under the action of $u\partial$ because $(u\partial x)^{(i)} = u \partial(x^{(i)})$ for $x \in K$.
(If $K$ does not contain a primitive $p$-th root of unity, we may still define $K^{(\partial)}$ using Galois descent.)

We call the inclusion $\varphi^{(\partial)*}: K^{(\partial)} \inj K$ 
the \emph{$\partial$-Frobenius morphism}.
We view $K^{(\partial)}$ as being equipped with the derivation
$\partial' = \partial / (pu^{p-1})$; we will see below
(Lemma~\ref{L:Kj-satisfies-hypo}) that $\partial'$ is of rational type with parameter $u^p$.

It is worth pointing out that $K^{(\partial)}$ depends on the choice of the rational parameter $u$, not just the derivation $\partial$.

Occasionally, we use $K^{(\partial,n)}$ to denote the subfield of $K$ obtained by applying the above construction $n$ times; if $K$ contains a primitive
$p^n$-th root of unity, this is the same as the fixed field for the natural
action of $\ZZ/p^n \ZZ$ on $K$.
\end{construction}

\begin{lemma} \label{L:norm of Frob derivative}
We have $|\partial'|_{K^{(\partial)}} = |u|^{-p}$.
\end{lemma}
\begin{proof}
We may assume that $K$ contains a primitive $p$-th root of unity $\zeta_p$.  We need only show that $u^p\partial'$ preserves $\gotho_{K^{(\partial)}}$.  
For any $x \in \gotho_{K^{(\partial)}}$, we have
$$
x = \frac 1p (x + x^{(1)} + \cdots + x^{(p-1)}) = \frac 1p \sum_{n =0}^\infty \frac{\partial^n(x)} {n!} u^n \sum_{i = 0}^{p-1}(\zeta_p^i - 1)^n.
$$
Applying $u^p\partial' = u\partial / p$ gives
\begin{eqnarray}
\nonumber u^p\partial'(x) &=& \frac u{p^2} \sum_{n=0}^\infty \left( \frac{\partial^{n+1}(x)} {n!} u^n\sum_{i = 0}^{p-1}(\zeta_p^i - 1)^n + \frac{\partial^n(x)} {(n-1)!} u^{n-1}\sum_{i = 0}^{p-1}(\zeta_p^i - 1)^n\right)\\
\label{E:partial-on-Kj} &=& \frac u{p^2} \sum_{n=0}^\infty \frac{\partial^{n+1}(x)} {n!} u^n \sum_{i = 0}^{p-1} (\zeta_p^i - 1)^n\zeta_p^i.
\end{eqnarray}
The sum $\sum_{i = 0}^{p-1} (\zeta_p^i - 1)^n\zeta_p^i$ equals 0 for
$n=0,\dots,p-2$; it equals $p$ for $n=p-1$; and it is 
a multiple of $p^2$ for any $n \geq p$
(because the quantity belongs both to $\ZZ$ and to the ideal
$(\zeta_p-1)^p$ in $\ZZ[\zeta_p]$).
Hence by \eqref{E:partial-on-Kj}, $u^p \partial'(x)$ equals
$u^p \partial^p(x)/p!$ plus an element of $\gotho_K$,
yielding $u^p \partial'(x) \in \gotho_K \cap K^{(\partial)} = \gotho_{K^{(\partial)}}$.
\end{proof}

\begin{lemma}\label{L:Kj-satisfies-hypo}
The differential operator $\partial'$ on $K^{(\partial)}$ is of rational
type, with parameter $u^p$.
\end{lemma}
\begin{proof}
Write
\begin{eqnarray*}
\frac {u^{pn}\partial'^n}{n!}(x) &=& \frac {(u^p\partial')(u^p \partial'-1) \cdots (u^p \partial' - (n-1))}{n!}(x) \\
&=& \frac {(u\partial)(u \partial - p) \cdots (u \partial - (n-1)p)}{n! \cdot p^n}(x)
\end{eqnarray*}
As a corollary of Lemma~\ref{L:norm of Frob derivative},
for any element $x \in K^{(\partial)}$ and $i \in \ZZ \setminus p\ZZ$, $|(u \partial - i)(x)| = |x|$.  Since $u\partial$ fixes $K^{(\partial)}$,
applying differential operators $u\partial - i$ for $i \in \ZZ \bs p\ZZ$ to the result will not change the norm, so
$$
\left|\frac {u^{pn}\partial'^n}{n!}(x) \right| = \left| \frac {(u\partial)(u \partial - 1) \cdots (u \partial - (np-1))}{n! \cdot p^n}(x) \right| = \left| \frac {u^{np}\partial^{np}} {(np)!} (x) \right|.
$$
The statement follows.
\end{proof}

\begin{defn}
Given a $\partial'$-differential module $V'$ over $K^{(\partial)}$, 
we may view $\varphi^{(\partial)*} V' = V' \otimes_{K^{(\partial)}} K$ 
as a $\partial$-differential module over $K$ by setting
$$
\partial (\bbv' \otimes x) = pu^{p-1} \partial'(\bbv') \otimes x + \bbv' \otimes \partial(x) \qquad (\bbv' \in V', x \in K).
$$
\end{defn}

\begin{lemma}\label{L:Frob-pull-IR-ineq}
Let $V'$ be a $\partial'$-differential module over $K^{(\partial)}$.  Then
$$
IR_\partial(\varphi^{(\partial)*} V') \geq \min \{IR_{\partial'}(V')^{1/p}, p\, IR_{\partial'}(V')\}.
$$
\end{lemma}
\begin{proof}
This is essentially \cite[Lemma~10.3.2]{kedlaya-course}.
Consider the diagram
$$
\xymatrix{
K^{(\partial)} \ar[rr]^-{f'^*_\gen} \ar[d]^{\varphi^{(\partial)*}} && 
K^{(\partial)} \llbracket T' / u^p \rrbracket_0 \ar[d]^{\tilde \varphi^{(\partial)*}} \\
K \ar[rr]^-{f^*_\gen} && K \llbracket T / u \rrbracket_0
}
$$
where $\tilde \varphi^{(\partial)*}$ is a $K^{(\partial)}$-homomorphism extending $\varphi^{(\partial)*}$ by $\tilde \varphi^{(\partial)*} (T') = (u + T)^p - u^p$.  The diagram commutes because formally
\begin{eqnarray*}
(\tilde{\varphi}^{(\partial)*} \circ f'^*_\gen)(x) &=& \tilde{\varphi}^{(\partial)*} \left(
\sum_{n=0}^\infty \binom{u^p\partial'}{n} (x) \left(\frac {T'}{u^p}\right)^n
\right) = \sum_{n=0}^\infty \binom{u\partial/p}n (\varphi^{(\partial)*}(x)) \left(\left( 1+ \frac Tu \right)^p -1 \right)^n \\
&=& \left(1+ \left( 1+ \frac Tu \right)^p -1\right)^{u\partial/p}(\varphi^{(\partial)*}(x)) = \left(\left(1+\frac Tu\right)^p\right)^{u\partial/p}(\varphi^{(\partial)*}(x)) \\
&=& \left(1+ \frac Tu\right)^{u\partial}(\varphi^{(\partial)*}(x))
= \sum_{n=0}^\infty \binom{u\partial}{n} (\varphi^{(\partial)*}(x)) \left(\frac Tu \right)^n = (f^*_\gen \circ \varphi^{(\partial)*})(x).
\end{eqnarray*}
For $x \in K^{(\partial)}$, all of the series in this formal equation
converge, and we obtain correct equalities.

For $r' \in [0,1)$, set $r = \min \{ (r')^{1/p}, pr'\}$, or equivalently, $r' = \max \{ r^p, p^{-1}r\}$.  By Proposition~\ref{P:generic-point},
\begin{eqnarray*}
&& R_\partial(V') \geq r'|u|^p \\
&\LRar& f'^*_\gen V' \textrm{ is a trivial $\partial_{T'}$-differential module over } A_{K^{(\partial)}} [0, r'|u|^p)\\
&\Rar& \tilde{\varphi}^{(\partial)*} f'^*_\gen V' =
f^*_\gen \varphi^{(\partial)*} V'
\textrm{ is a trivial $\partial_T$-differential module over } A_K [0, r |u|)\\
&\LRar& R_\partial(\varphi^{(\partial)*}V') \geq r|u|.
\end{eqnarray*}
where the second implication is a direct corollary of the lemma below.  The statement follows.
\end{proof}

\begin{lemma}
\cite[Lemma~10.2.2]{kedlaya-course} Let $K$ be a nonarchimedean field. 
For $u,T \in K$ and $r \in (0,1)$, if $|u - T| < r |u|$, then
$$
|u^p - T^p| \leq \max \{ r^p|u|^p, p^{-1}r|u|^p \}.
$$
\end{lemma}

\begin{defn}
For a $\partial$-differential module $V$ over $K$, define the \emph{$\partial$-Frobenius descendant} of $V$ as the $K^{(\partial)}$-module
$\varphi^{(\partial)}_* V$ obtained from $V$ by restriction along $\varphi^{(\partial)*}: K^{(\partial)} \rar K$, viewed as a $\partial'$-differential module over $K^{(\partial)}$ with differential $\partial' = \frac 1{pu^{p-1}} \partial$.  Note that this operation commutes with duals.
\end{defn}

\begin{defn}
For $n = \seriezero{}p-1$, let $W_n^{(\partial)}$ be the $\partial'$-differential module over $K^{(\partial)}$ with one generator $\bbv$, such that
$$
\partial'(\bbv) = \frac np u^{-p} \bbv.
$$
{}From the Newton polynomial associated to $\bbv$, we read off $IR_{\partial'}(W_n^{(\partial)}) = p^{-p/(p-1)}$ for $n \neq 0$.  (One may view the generator $\bbv$ as a proxy for $u^n$.)
\end{defn}

\begin{lemma}\label{L:Frob-pullback-vs-desc}
We have the following relations between $\partial$-Frobenius pullbacks and $\partial$-Frobenius descendants.
\begin{enumerate}
\item[(a)]
For $V$ a $\partial$-differential module over $K$, there are canonical isomorphisms
$$
\iota_n : (\varphi^{(\partial)}_* V) \otimes W_n^{(\partial)} \simeq \varphi^{(\partial)}_* V \qquad (n = \seriezero{}p-1).
$$
\item[(b)]
For $V$ a $\partial$-differential module over $K$, a submodule $U$ of $\varphi^{(\partial)}_* V$ is itself the $\partial$-Frobenius descendant of a submodule of $V$ if and only if $\iota_n(U \otimes W_n^{(\partial)}) = U$ for $n = \seriezero{}p-1$.
\item[(c)]
For $V$ a $\partial$-differential module over $K$, there is a canonical isomorphism
$$
\varphi^{(\partial)*}\varphi^{(\partial)}_* V \simeq V^{\oplus p}.
$$
\item[(d)]
For $V'$ a $\partial'$-differential module over $K^{(\partial)}$, there is a canonical isomorphism
$$
\varphi^{(\partial)}_* \varphi^{(\partial)*} V' \simeq \bigoplus_{n = 0}^{p-1} (V' \otimes W_n^{(\partial)}).
$$
\item[(e)]
For $V_1$, $V_2$ $\partial$-differential modules over $K$, there is a canonical isomorphism
$$
\varphi^{(\partial)}_* V_1 \otimes \varphi^{(\partial)}_* V_2 \simeq \bigoplus_{n = 0}^{p-1} W_n^{(\partial)} \otimes \varphi^{(\partial)}_* (V_1 \otimes V_2).
$$
\item[(f)]
For $V$ a $\partial$-differential module over $K$, there are canonical bijections
$$
H^i_{\partial}(V) \simeq H^i_{\partial'}(\varphi^{(\partial)}_* V) \qquad  (i = 0, 1).
$$
\end{enumerate}
\end{lemma}
\begin{proof}
Straightforward.
\end{proof}

\begin{defn}
Let $V$ be a $\partial$-differential module over $K$ such that $IR_\partial(V) > p^{-1/(p-1)}$. A \emph{$\partial$-Frobenius antecedent} of $V$ is a $\partial'$-differential module $V'$ over $K^{(\partial)}$ such that $V \simeq \varphi^{(\partial)*}V'$ and $IR_{\partial'}(V') > p^{-p/(p-1)}$.
\end{defn}

\begin{prop}[Christol-Dwork]
\label{P:CD-Frob-ante}
Let $V$ be a $\partial$-differential module over $K$ such that $IR_\partial(V) > p^{-1/(p-1)}$. Then there exists a unique $\partial$-Frobenius antecedent $V'$ of $V$.  Moreover, $IR_{\partial'}(V') = IR_\partial(V)^p$.
\end{prop}
\begin{proof}
As in \cite[Theorem~10.4.2]{kedlaya-course}.
\end{proof}

\begin{remark}
As in \cite[Theorem~10.4.4]{kedlaya-course}, one can form a version of 
Proposition~\ref{P:CD-Frob-ante}
for differential modules over discs and annuli.
\end{remark}

\begin{theorem}\label{T:Frob-push-sp-norm}
Let $V$ be a $\partial$-differential module over $K$.  Then
$$
\gothI\gothR_{\partial'}(\varphi^{(\partial)}_*V) = 
\bigcup_{r \in \gothI\gothR_\partial(V)} \left\{
\begin{array}{ll}
\{r^p,\ p^{-p/(p-1)}\ (p-1 \textrm{ times})\} & r > p^{-1/(p-1)}\\
\{p^{-1} r\ (p \textrm{ times})\} & r \leq p^{-1/(p-1)}.
\end{array}
\right.
$$
In particular, $IR_{\partial'}(\varphi^{(\partial)}_* V) = \min \{p^{-1} IR_\partial(V),\ p^{-p/(p-1)}\}$.
\end{theorem}
\begin{proof}
The proof is identical to that of  \cite[Theorem~10.5.1]{kedlaya-course}.
\end{proof}

\begin{corollary} \label{C:pullback radius}
Let $V'$ be a $\partial'$-differential module over $K^{(\partial)}$ such that
$IR_{\partial'}(V') \neq p^{-p/(p-1)}$. 
Then $IR_\partial(\varphi^{(\partial)*} V') = \min\{IR_{\partial'}(V')^{1/p}, p \, IR_{\partial'}(V')\}$.
\end{corollary}
\begin{proof}
In case $IR_{\partial'}(V') > p^{-p/(p-1)}$, this holds by \cite[Corollary~10.4.3]{kedlaya-course}. Otherwise, by Lemma~\ref{L:Frob-pullback-vs-desc}(d), 
$\varphi_*^{(\partial)} \varphi^{(\partial)*} V'
\cong \oplus_{m=0}^{p-1} (V' \otimes W_m^{(\partial)})$ and $IR_{\partial'}(V' \otimes W_m^{(\partial)})
= IR_{\partial'}(V')$ since $IR_{\partial'}(V') < IR_{\partial'}(W_m^{(\partial)})$. Hence by
Theorem~\ref{T:Frob-push-sp-norm},
\[
IR_{\partial'}(V') = IR_{\partial'}(\varphi^{(\partial)}_* \varphi^{(\partial)*} V') = \min\{p^{-1} IR_\partial(\varphi^{(\partial)*} V'),
p^{-p/(p-1)}\}.
\]
We get a contradiction if the right side equals $p^{-p/(p-1)}$,
so we must have $IR_{\partial'}(V') = p^{-1} IR_\partial(\varphi^{(\partial)*} V') \leq p^{-p/(p-1)}$, 
proving the claim.
\end{proof}

For the following theorem, we do not assume $p>0$.
\begin{theorem}\label{T:decomp-over-field-complete}
Let $V$ be a $\partial$-differential module over $K$.  Then there exists a decomposition
$$
V = \bigoplus_{r \in (0, 1]}V_r,
$$
where every subquotient of $V_r$ has pure intrinsic $\partial$-radii $r$.  
Moreover, if $p=0$, then $r^{\dim V_r} \in |K^\times|$;
if $p>0$, then for any nonnegative integer $h$, we have
$$
r < p^{-p^{-h}/(p-1)} \quad \Longrightarrow \quad r^{\dim V_r} \in 
|(K^{(\partial, h)})^\times|^{p^{-h}}.
$$
\end{theorem}
\begin{proof}
The proof is similar to those of  \cite[Theorem~10.6.2]{kedlaya-course} and \cite[Theorem~10.7.1]{kedlaya-course}.
\end{proof}

\begin{remark}\label{R:partial_j-no-integrality}
In the case when $K$ is the completion of $K_0(u)$ with respect to the $\eta$-Gauss norm, 
$K^{(\partial, h)}$ is the completion of $K_0(u^{p^h})$ with respect to the
$\eta^{p^h}$-Gauss norm.  
We deduce thus from Theorem~\ref{T:decomp-over-field-complete} that
$r^{\dim V_r} \in |K_0^\times|^{p^{-h}}\eta^\ZZ$.
\end{remark}

\begin{remark}
Let $K'$ be a complete extension of $K$ equipped with an extension
of $\partial$ which is again of rational type with parameter $u$.
Then the intrinsic radii of a $\partial$-differential module over $K$
are the same as that of its base extension to $K'$: namely,
this is clear from Remark~\ref{R:part-decomp-over-field} for those radii
less than $\omega$, but we can reduce to this case using
Theorem~\ref{T:Frob-push-sp-norm}.
\end{remark}

\subsection{Multiple derivations}
\label{S:multi-derivations}

In this subsection, we introduce differential fields of higher order.

\begin{defn}
Let $K$ denote a differential ring of order $m$, i.e., a ring $K$ equipped with $m$ commuting derivations $\serie{\partial_}m$.  For $j \in J = \{1, \dots, m\}$, a \emph{$\partial_j$-differential module} is a finite projective $K$-module $V$ equipped with the action of $\partial_j$.  In other words, we view $K$ as a differential ring of order 1 by forgetting the derivations other than
$\partial_j$.  A \emph{$(\serie{\partial_}m)$-differential module} 
(or \emph{$\partial_J$-differential module},
or simply a \emph{differential module})
is a finite projective $K$-module $V$ equipped with commuting 
actions of $\serie{\partial_}m$.  We may apply the results above by singling out one of $\serie{\partial_}m$.
\end{defn}

\begin{defn}\label{D:multi-op-cvgt-radii}
Let $K$ be a complete nonarchimedean
differential field of order $m$ and characteristic zero,
and let $V$ be a nonzero $(\serie{\partial_}m)$-differential module over $K$.
Define the \emph{intrinsic generic radius of convergence},
or for short the \emph{intrinsic radius}, of $V$ to be
$$
IR(V) = \min_{j \in J} \left\{IR_{\partial_j}(V) \right\} = \min_{j \in J} \left\{ |\partial_j|_{\sp, K} \big/ |\partial_j|_{\sp, V}\right\}.
$$
For $j \in J$, we say $\partial_j$ is \emph{dominant} for $V$ if 
$IR_{\partial_j}(V) = IR(V)$.  We define the \emph{intrinsic subsidiary radii} 
$\gothI\gothR(V) = \{IR(V;1), \dots, IR(V;\dim V)\}$ by collecting and ordering 
intrinsic radii from Jordan-H\"older factors, as in
Definition~\ref{D:cvgt-radii}. We again say that $V$ has \emph{pure intrinsic
radii} if the elements of $\gothI\gothR(V)$ are all equal to $IR(V)$.
\end{defn}

\begin{defn} \label{D:rational type multi}
Let $K$ be a complete nonarchimedean differential field of order $m$ and
characteristic zero.
We say that $K$ is of \emph{rational type} with respect to a set of parameters
$\{u_j: j \in J\}$ if each $\partial_j$ is of rational type with respect to
$u_j$, and $\partial_i(u_j) = 0$ for $i \neq j$.
\end{defn}

\begin{remark}\label{R:enlarge-K}
Set notation as in Definition~\ref{D:rational type multi}.
Let $K'$ be the completion of $K(t)$ for the $\eta$-Gauss norm; then
$K'$ is again of rational type with respect to $u_1,\dots,u_m,t$.
\end{remark}

\begin{remark}
Recall that if $p>0$, we have a $\partial_j$-Frobenius $\varphi^{(\partial_j)*}: K^{(\partial_j)} \inj K$ for $j \in J$.  Since the elements $u_{J\setminus \{j\}}$ are killed by $\partial_j$, they are elements in $K^{(\partial_j)}$.  Hence by Lemma~\ref{L:Kj-satisfies-hypo}, the differential operators $\partial_{J\setminus \{j\}}$ and $\partial'_j$ are of rational type over $K^{(\partial_j)}$ 
with respect to the parameters  $u_{J \setminus\{j\}}$ and $u_j^p$.
\end{remark}

\begin{theorem}\label{T:multi-decomp-over-field}
Let $K$ be a complete nonarchimedean differential field of order $m$ and characteristic zero, of rational type.
Let $V$ be a $\partial_J$-differential module over $K$.  Then there exists a decomposition
$$
V = \bigoplus_{r \in (0, 1]}V_r,
$$
where every subquotient of $V_r$ has pure intrinsic radii $r$.
Moreover, if $p=0$, then $r^{\dim V_r} \in |K^\times|$;
if $p>0$, then 
\[
r < p^{-p^{-h}/(p-1)} \quad \Longrightarrow \quad
r^{\dim V_r} \in |K^{\times}|^{1/p^h}.
\]
\end{theorem}
\begin{proof}
Since the $\partial_J$ commute with each other, the theorem follows by applying Theorem~\ref{T:decomp-over-field-complete} to each $\partial_j$ and forming
a common refinement of the resulting decompositions.
\end{proof}

\begin{defn}
For $l/k$ an extension of fields of characteristic $p>0$, we say the extension is \emph{separable} if $l \cap k^{p^{-1}} = k$.  A \textit{$p$-basis} of $l$ over $k$ is a set $B = \{u_j\}_{j \in J} \subset l$ such that the products $u_J^{e_J}$, where $e_j \in \{0, 1, \dots, p-1\}$ for all $j \in J$ and $e_j = 0$ for all but finitely many $j$, form a basis of the vector space $l$ over $kl^p$. By a \textit{$p$-basis} of $l$ we mean a $p$-basis of $l$ over $l^p$.  (For more details, see \cite[p. 565]{eisenbud} or 
\cite[Ch.0,~\S21]{ega41}.)

For an extension $L/K$ of complete nonarchimedean fields with residue fields $l,k$
of characteristic $p>0$, with $l/k$ separable, 
a \emph{$p$-basis} of $L$ over $K$ will mean a set of elements $(u_J) \subset \gotho_L^\times$ whose images $(\bar u_J) \subset l$ form a $p$-basis of $l$ over $k$.
\end{defn}

One important instance of Definition~\ref{D:rational type multi}
is the following.

\begin{situation}\label{Sit:K-in-SwanI}
Let $m$ be a nonnegative integer and $J = \{\serie{}m\}$.  
Let $F$ be a complete discrete valuation field of characteristic $0$
with residue field $\kappa$ of characteristic $p>0$.  
Let $K_1$ be a complete extension of $F$ with the same value group and residue
field $k_1$ separable over $\kappa$.
Assume $K_1$ has a finite $p$-basis $(\serie{u_}{m})$ over $F$.  
Let $F'$ be an extension of $F$ complete for a (not necessarily discrete) 
nonarchimedean norm $|\cdot|$, with the same residue field $\kappa$.  
Let $K_2$ be the completion of $K_1 \otimes_F F'$.  
Let $k$ be a (possibly infinite) separable algebraic
extension of $k_1$, 
and let $K$ be the completion of the unramified extension of $K_2$ with residue field $k$.
\end{situation}

\begin{lemma}\label{L:K-in-SwanI}
In Situation~\ref{Sit:K-in-SwanI}, the natural projection $\Omega^1_K \surj \bigoplus_{j=1}^m K \cdot du_j$ gives derivations
$(\partial_j = \partial_{u_j})_{j \in J}$ of rational type with respect to $u_1,\dots,
u_m$.
\end{lemma}
\begin{proof}
It is enough to check for $K_1$: it is clear that the same conclusion
then holds for $K_2$, 
and then Lemma~\ref{L:H-under-unram-extn} implies
the same conclusion for $K$. 
That is, we must check that
$\gotho_{K_1}$ is stable under $\partial_j^n / n!$ for all nonnegative
integers $n$ and all $j \in J$.
For each $n \in \NN$, any element $x \in \gotho_{K_1}$ can be written (not uniquely) as
$$
x = \sum_{i = 0}^{+\infty} \sum_{e_J = 0}^{p^n-1} \alpha_{n ,i, e_J}^{p^n} u_J^{e_J} \pi_F^i,
$$
where $\alpha_{n ,i, e_J} \in \gotho_{K_1}^\times \cup \{0\}$.  Then for any $j_0 \in J$,
$$
\frac {\partial_{j_0}^n} {n!} (x) = \sum_{i = 0}^{+\infty} \sum_{e_J = 0}^{p^n-1} \sum_{\beta = 0}^n \frac {\partial_{j_0}^\beta}{\beta!} \left(\alpha_{n ,i, e_J}^{p^n}
\right) \frac {\partial_{j_0}^{n-\beta}} {(n-\beta)!} \left(u_J^{e_J}\right) \pi_F^i \in \gotho_{K_1}.
$$
The lemma follows.
\end{proof}

\begin{remark} \label{R:includes swan1}
Situation~\ref{Sit:K-in-SwanI} includes the two options in
\cite[Hypothesis~2.1.3]{kedlaya-swan1}.
(Note that \cite[Hypothesis~2.1.3(b)]{kedlaya-swan1} should require 
that $l/k$ be separable.)
We will see later 
(Theorem~\ref{T:swan1})
that the results in \cite{kedlaya-swan1} carry over to differential
fields of rational type.
\end{remark}

\section{Differential modules on one-dimensional spaces}
\setcounter{equation}{0}

Having considered differential modules over fields, we next consider
differential modules on a disc or annulus over a differential field.
This parallels \cite[Chapters~11 and~12]{kedlaya-course}.

\setcounter{theorem}{0}
\begin{hypo}\label{H:K}
Throughout this section, we assume that $K$ is a 
complete (not necessarily discretely valued) nonarchimedean differential field of order $m$, characteristic zero, and residual characteristic $p$
(not necessarily positive). We also assume $K$ is of rational type.
\end{hypo}

\begin{notation}\label{N:K-m-J}
Let $\serie{\partial_}m$ denote the derivatives on $K$
and let $\serie{u_}m$ denote a set of corresponding rational parameters.
Let $J = \{\serie{}m\}$.  We reserve $j$ and $J$ for indexing derivations.
\end{notation}

\subsection{Setup}
\label{S:setup-multi-dim}

\begin{notation}
For $\eta > 0$, let $F_\eta$ be the completion of $K(t)$ under the 
$\eta$-Gauss norm $|\cdot|_\eta$. Put $\partial_0 = \frac d {dt}$ on $F_\eta$;
by Remark~\ref{R:enlarge-K}, $F_\eta$ is of rational type for the derivations
$\partial_{J^+}$, where $J^+ = J \cup \{0\} = \{0,\dots,m\}$.
\end{notation}

\begin{remark}
For $I \subseteq [0, +\infty)$ an interval and $j \in J^+$, 
we may refer to differential modules or 
$\partial_{j}$-differential modules over
$A_K^1(I)$, meaning locally free coherent sheaves with the appropriate
derivations.
For $I=[\alpha,\beta]$ closed, these are just modules with appropriate
derivations over the principal ideal domain $K \langle \alpha/t,
t/\beta \rangle$; in particular,
any $\partial_j$-differential module over a closed annulus is free by
\cite[Proposition~9.1.2]{kedlaya-course}.
\end{remark}

\begin{remark}
For $I \subseteq [0, +\infty)$ an interval, and 
$M$ a nonzero $\partial_j$-differential module over $A_K^1(I)$,
it is unambiguous to refer to the intrinsic $\partial_j$-radius 
of convergence $IR_{\partial_j}(M \otimes F_\eta)$ 
of $M$ at $|t| = \eta$.
\end{remark}

The intrinsic radii are stable under tame base change.

\begin{prop}\label{P:tame-base-change}
Let $n$ be a (possibly negative) nonzero integer (coprime to $p$ if $p>0$), and let $f_n^* : F_\eta \rar F_{\eta^{1/n}}$ be the map $t \rar t^n$. Then for any $j \in J^+$, and for any $\partial_j$-differential module $V$ over $F_\eta$, $IR_{\partial_j}(V) = IR_{\partial_j}(f_n^*V)$ and hence $\gothI\gothR_{\partial_j}(V) = \gothI\gothR_{\partial_j}(f_n^*V)$.  
\end{prop}
\begin{proof}
The proof for $j = 0$ is in \cite[Proposition~9.7.6]{kedlaya-course},
 and the proof for $j \in J$ is to apply
Remark~\ref{R:base-change-sp-norm}.
\end{proof}
\begin{remark}
One may also consider off-centered tame base change, 
as in \cite[Exercise~9.8]{kedlaya-course}.
\end{remark}

\subsection{Variation of subsidiary radii}
\label{variation-partial_j}

In this subsection, we prove slightly weakened analogues of some results in 
\cite[Chapter~11]{kedlaya-course}.  
We begin by studying the variation of slopes of Newton polygons.

\begin{notation}
Let $P \in K \langle \alpha / t, t / \beta \rangle[T]$ be a polynomial of degree $d$.  For $r \in [-\log \beta, -\log \alpha]$, let $\NP_r(P)$ denote the Newton polygon of $P$ under $|\cdot|_{e^{-r}}$.
\end{notation}

\begin{prop}\label{P:newton}
For $r \in [-\log \beta, -\log \alpha]$, 
let $f_1(P, r), \dots, f_d(P,r)$ be the slopes of $\NP_r(P)$ in increasing 
order.  For $i = \serie{}d$, put $F_i(P, r) = f_1(P, r) + \cdots + f_i(P, r)$.
\begin{enumerate}
\item[(a)]
(Linearity) For $i = \serie{}d$, the functions $f_i(P, r)$ and $F_i(P, r)$ are continuous and piecewise affine in $r$.
\item[(b)]
(Integrality) If $i = d$ or $f_i(r_0) < f_{i+1}(r_0)$, then the slopes of $F_i(P, r)$ in some neighborhood of $r = r_0$ belong to $\ZZ$.  
Consequently, the slopes of each $f_i(P, r)$ and $F_i(P, r)$ belong to $\frac 11 \ZZ \cup \cdots \cup \frac 1d \ZZ$.
\item[(c)]
(Monotonicity)  Suppose that $P$ is monic and $\alpha = 0$.
For $i = \serie{}d$, the slopes of $F_i(P, r)$ are nonnegative.
\item[(d)]
(Concavity)  Suppose that $P$ is monic.  For $i = \serie{}d$, the function $F_i(P, r)$ is concave.
\item[(e)]
(Truncation) For any fixed $a \in \RR^+$ and $b \in \RR$, the statements (a), (c), and (d) are also true if we replace $f_i(P, r)$ by $\min \{f_i(P, r), ar+b\}$ for all $i \in \{\serie{}d\}$.
\end{enumerate}
\end{prop}
\begin{proof}
See \cite[Theorem~11.2.1]{kedlaya-course} and \cite[Remark~11.2.4]{kedlaya-course}.
\end{proof}

\begin{lemma}[Lattice lemma] \label{L:lattice} 
Put $R = K \langle t \rangle$,
$\cup_{\alpha < 1} K \langle \alpha/t, t \rangle$, or
$\cup_{\alpha < 1 < \beta} K \langle \alpha/t, t/\beta \rangle$,
or (if $K$ is discrete)
$K \llbracket t \rrbracket_0$
or $\cup_{\alpha < 1} K \llbracket \alpha/t, t \rrbracket_0$
equipped with the norm $|\cdot|_1$.
Let $M$ be a finite free $R$-module of rank $n$, 
and let $|\cdot|_M$ be a norm on $M$ compatible with $R$.  Assume that either:
\begin{enumerate}
\item[(a)]
$c>1$, and the value group of $K$ is not discrete; or
\item[(b)]
$c\geq 1$, and the value groups of $K$ and $M$ coincide and are discrete.
\end{enumerate}
Then there exists a basis of $M$ defining a supremum norm $|\cdot|'_M$ for which $c^{-1} |m|_M \leq |m|'_M \leq c |m|_M$ for $m \in M$.
\end{lemma}
\begin{proof}
Let $F$ be the completion of $\Frac R$ under $|\cdot|_1$.
By \cite[Lemma~1.3.7]{kedlaya-course},
we can construct a basis of $M \otimes F$ defining a supremum norm $|\cdot|'_M$
for which $c^{-1} |m|_M \leq |m|'_M \leq c |m|_M$ for $m \in M$.
If $R = K \langle t \rangle$,
or $K$ is discrete and $R = K \llbracket t \rrbracket_0$,
then \cite[Lemma~8.6.1]{kedlaya-course} gives
a basis of $M$ defining the same supremum norm $|\cdot|'_M$.
If $R = \cup_{\alpha < 1} K \langle \alpha/t, t \rangle$ or
$\cup_{\alpha < 1 < \beta} K \langle \alpha/t, t/\beta \rangle$,
then \cite[Lemma~8.6.1]{kedlaya-course} gives a basis of
$K \langle 1/t, t \rangle$ defining $|\cdot|'_M$.
However, we can approximate that basis arbitrarily closely
with a basis of $M$ itself,
because  $R$ is dense in $K \langle 1/t, t \rangle$ under $|\cdot|_1$,
and any element of $R$ with an inverse in $K \langle 1/t, t \rangle$
also has an inverse in $R$. Any sufficiently good approximation
will define the same supremum norm.
If $K$ is discrete and 
$R = \cup_{\alpha < 1} K \llbracket \alpha/t, t \rrbracket_0$,
then $R$ itself is a field, so we can approximate a basis of $M \otimes F$
with a basis of $M$ defining the same supremum norm.
\end{proof}

\begin{notation}\label{N:subsidiary-radii-and-sums}
Fix $j \in J^+$.  Let $M$ be a $\partial_j$-differential module of rank $d$
over $K \langle \alpha/t, t/\beta \rangle$.  For $r \in [-\log \beta, -\log \alpha]$ and $i \in \{\serie{}d\}$, define 
$$
f_i^{(j)}(M,r) = -\log R_{\partial_j}(M \otimes F_{e^{-r}}; i), \quad F_i^{(j)}(M,r) = f_1^{(j)}(M,r) + \cdots + f_i^{(j)}(M,r).
$$
\end{notation}

\begin{theorem} \label{T:subsidiary-partial-0}
\cite[Theorem~11.3.2]{kedlaya-course}
Let $M$ be a $\partial_0$-differential module of rank $d$ over $K \langle \alpha/t, t/\beta \rangle$.
\begin{enumerate}
\item[(a)]
(Linearity)
For $i=\serie{}d$, the functions $f_i^{(0)}(M,r)$ and $F_i^{(0)}(M,r)$ are continuous and piecewise affine.
\item[(b)]
(Integrality)
If $i=d$ or $f_i^{(0)}(M,r_0) > f_{i+1}^{(0)}(M,r_0)$, then 
the slopes of $F_i^{(0)}(M,r)$ in some neighborhood of $r_0$ belong to $\ZZ$. Consequently, the slopes of each $f_i^{(0)}(M,r)$ and $F_i^{(0)}(M,r)$ belong to
$\frac{1}{1} \ZZ \cup \cdots \cup \frac{1}{d} \ZZ$.
\item[(c)]
(Monotonicity)
Suppose that  $\alpha = 0$.  For any point $r_0$ where $f_i^{(0)}(M,r_0) > r_0$, the slopes of $F_i^{(0)}(M,r)$ are nonpositive in some neighborhood of $r_0$.  Also, $f_i^{(0)}(M , r_0) = r_0$ for $r_0$ sufficiently large.
\item[(d)]
(Convexity)
For $i=1, \dots, d$, the function $F_i^{(0)}(M,r)$ is convex.
\end{enumerate}
\end{theorem}

We have a similar but slightly weaker result for 
$\partial_j$-differential modules when $j \in J$.

\begin{theorem} \label{T:subsidiary-partial-j}
Fix $j \in J$.  Let $M$ be a $\partial_j$-differential module of rank $d$
over $K \langle \alpha/t, t/\beta \rangle$.
\begin{enumerate}
\item[(a)]
(Linearity)
For $i=\serie{}d$, the functions $f_i^{(j)}(M,r)$ and $F_i^{(j)}(M,r)$ are continuous.  They are piecewise affine in the locus where $f_i^{(j)}(M, r) > -\log |u_j|$; if $p=0$, they are in fact piecewise affine everywhere.
\item[(b)]
(Weak integrality)
\begin{enumerate}
\item[(i)] Suppose $p=0$.
If $i=d$ or $f_i^{(j)}(M,r_0) > f_{i+1}^{(j)}(M,r_0)$, 
then the slopes of $F_i^{(j)}(M,r)$ in some neighborhood of $r_0$ belong to 
$\ZZ$.  
Consequently, the slopes of each $f_i^{(j)}(M,r)$ and $F_i^{(j)}(M,r)$ at $r=r_0$ belong to
$\frac{1}{1} \ZZ \cup \cdots \cup \frac{1}{d} \ZZ$.
\item[(ii)] Suppose $p>0$.
If $i=d$ or $f_i^{(j)}(M,r_0) > f_{i+1}^{(j)}(M,r_0)$, and $f_i^{(j)} (M, r_0) > \frac 1 {p^n(p-1)} \log p - \log |u_j|$ for some $n \in \ZZ_{\geq 0}$, then the slopes of $F_i^{(j)}(M,r)$ in some neighborhood of $r_0$ belong to $\frac 1 {p^n}\ZZ$.  Consequently, if $f_i^{(j)} (M, r_0) > \frac 1 {p^n(p-1)} \log p - \log |u_j|$ for some $n \in \ZZ_{\geq 0}$, the slopes of each $f_i^{(j)}(M,r)$ and $F_i^{(j)}(M,r)$ at $r=r_0$ belong to
$\frac{1}{p^n} \ZZ \cup \cdots \cup \frac{1}{p^nd} \ZZ$.
\end{enumerate}
\item[(c)]
(Monotonicity)
Suppose that  $\alpha = 0$. For $i=1, \dots, d$, the slopes of $F_i^{(j)}(M,r)$ are nonpositive.
\item[(d)]
(Convexity)
For $i=1, \dots, d$, the function $F_i^{(j)}(M,r)$ is convex.
\end{enumerate}
\end{theorem}
\begin{proof}
We prove the theorem analogously to \cite[Theorem~11.3.2]{kedlaya-course}.  First 
of all, as in Remark~\ref{R:enlarge-K},
we may replace $K$ by the completion of $K(x)$ with respect to 
the
$|u_j|$-Gauss norm.
We may then replace $u_j$ by $u_j /x$ to reduce to the case $|u_j| = 1$.

We first show that the statements are true for $\tilde f_i^{(j)}(M, r) = \max \{f_i^{(j)}(M, r), \epsilon\}$ with $\epsilon > -\log \omega$ and $\widetilde F_i^{(j)}(M, r) = \tilde f_1^{(j)}(M, r) + \cdots + \tilde f_i^{(j)}(M, r)$.  Let $F = \Frac K \langle \alpha/t,t/\beta \rangle$. Choose a cyclic vector for $M \otimes F$ to obtain an isomorphism
$M \otimes F \cong F\{T\}/F\{T\}P$ for some monic twisted polynomial $P$ over $F$. We may then apply Proposition~\ref{P:newton} and Remark~\ref{R:part-decomp-over-field} to deduce (a) and (b), provided we omit the last assertion
in (a) (in case $p=0$); for that, see below.

For (c) and (d), 
it suffices to work in a neighborhood of some $r_0$.  Again by Remark~\ref{R:enlarge-K},
there is no harm in enlarging $K$ 
so that $e^{-r_0} \in |K^\times|$.  
We may reduce to the case $r_0 = 0$ by replacing $t$ by 
$\lambda t$ for some $\lambda \in K^\times$ with $|\lambda| = e^{-r_0}$.  
We then argue as in \cite[Lemma~11.5.1]{kedlaya-course} and deduce (c) and (d) from Proposition~\ref{P:newton}, as follows.  We may further enlarge $K$ to 
include $\lambda_{1},\dots, \lambda_{n} \in \ker(\partial_j)$
such that 
\[
-\log |\lambda_{j}| = \min \left\{-\log \omega - f_j(M,0), 0 \right\} \qquad (j=1,\dots,d).
\]
Let $B_0$ be the basis of $M \otimes F_{1}$ given by 
\[
\lambda_{d-1}^{-1}\cdots \lambda_{d-j}^{-1} T^j \qquad (j=0,\dots,d-1).
\]
Let $N_0$ be the characteristic polynomial of the matrix of action of $\partial_j$ on $B_0$.
Let $\mu_1, \dots, \mu_n$ be the eigenvalues of $N_0$, labeled so that
$|\mu_1| \geq \cdots \geq |\mu_n|$.
By \cite[Proposition~4.3.10]{kedlaya-course}, we
have $\max\{|\mu_j|, 1\} = \max\{\omega e^{f_j(M,0)},1\}$ for $j=1,\dots,d$. By Lemma~\ref{L:lattice}, for each $c>1$, we may construct a basis $B_c$ of $M$ such that the supremum norms $|\cdot|_0, |\cdot|_c$ defined by $B_0, B_c$ satisfy $c^{-1} |\cdot|_c \leq |\cdot|_0 \leq c |\cdot|_c$.  Let $N_c$ be the matrix of action of $\partial_j$ on $B_c$.  For $c > 1$ sufficiently small, \cite[Theorem~6.7.4]{kedlaya-course} implies that for $r$ close to 0, the visible spectrum of $M \otimes F_{e^{-r}}$ 
is the multiset of those
norms of eigenvalues of the characteristic polynomial of $N_c$ which exceed
$1$.
(Here the \emph{visible spectrum} of $M \otimes F_{e^{-r}}$ is
defined as in \cite[Definition~6.5.1]{kedlaya-course},
i.e., those spectral norms of subquotients of 
$M \otimes F_{e^{-r}}$ which exceed 1, with appropriate multiplicities.)
We may then deduce (c) and (d) from Proposition~\ref{P:newton}(c) and (d).

We next relax the truncation condition that we have imposed; we may assume
$p>0$ as otherwise there is nothing to check.
For each nonnegative integer $n$, 
we prove the claim for $\tilde f_i^{(j)}(M, r) = \max \{f^{(j)}_i(M, r), \epsilon\}$ and $\widetilde F^{(j)}_i(M, r) = \tilde f^{(j)}_1(M, r) + \cdots + \tilde f^{(j)}_i(M, r)$ with $\epsilon \in \left ( \frac 1{p^n(p-1)} \log p, \frac 1{p^{n-1}(p-1)} \log p \right]$, by induction on $n$;
the base case $n=0$ is proved above.  As above, 
we may reduce to the case $r_0=0$.

Consider the $\partial_j$-Frobenius $\varphi^{(\partial_j)*}: F_{e^{-r}}^{(\partial_j)} \inj F_{e^{-r}}$.  Put
$g_i^{(j)}(r) = -\log R_{\partial_j}(\varphi^{(\partial_j)}_*M \otimes F_{e^{-r}}^{(\partial_j)}; i)$ and $\tilde g_i^{(j)}(r) = \max \{g_i^{(j)}(r), p\epsilon\}$ for $i = \serie{}pd$.  By Theorem~\ref{T:Frob-push-sp-norm},
the list $\{g_1^{(j)}(r), \dots, g_{pd}^{(j)}(r)\}$ consists of
\[
\bigcup_{i=1}^d \begin{cases}
\{p f^{(j)}_i(M,r), \frac{p}{p-1} \log p \mbox{ ($p-1$ times)}
\} & f_i^{(j)}(M,r) \leq \frac 1{p-1} \log p \\
\{\log p + f_i^{(j)}(M,r) \mbox{ ($p$ times)}
\} & f_i^{(j)}(M,r) \geq \frac 1{p-1} \log p.
\end{cases}
\]
Thus, the list $\tilde g_1^{(j)}(r), \cdots, \tilde g_{pd}^{(j)}(r)$ consists of
\[
\bigcup_{i=1}^d \begin{cases}
\{p \tilde f^{(j)}_i(M,r), \frac{p}{p-1} \log p \mbox{ ($p-1$ times)}
\} & f_i^{(j)}(M,r) \leq \frac 1{p-1} \log p \\
\{\log p + \tilde f_i^{(j)}(M,r) \mbox{ ($p$ times)}
\} & f_i^{(j)}(M,r) \geq \frac 1{p-1} \log p.
\end{cases}
\]
We may thus
deduce (a) and (b) directly from the induction hypothesis. 
We similarly deduce (d) as
in \cite[Lemma~11.6.1]{kedlaya-course},
except that we are considering $\tilde g_i^{(j)}(r)$ but 
not $\tilde g_i^{(j)}(pr)$; this explains the weakened integrality
result. 
(See also Remark~\ref{R:partial_j-no-integrality}.)
Also, we can luckily deduce (c) directly, 
because $\varphi^{(\partial_j)*}$ does not
introduce a singularity on $A_K^1[0,\beta]$; by contrast, in the proof
of \cite[Theorem~11.3.2]{kedlaya-course}, one must switch to an off-centered
Frobenius to avoid a singularity at $t=0$.

We deduce that (a)-(d) hold for
$\tilde f^{(j)}_i(M, r) = \max \{f^{(j)}_i(M, r), \epsilon\}$ and $\widetilde F^{(j)}_i(M, r) = \tilde f^{(j)}_1(M, r) + \cdots + \tilde f^{(j)}_i(M, r)$ with $\epsilon>0$.  The desired results hold by taking $\epsilon \rar 0^+$.

This completes the proof except that if $p=0$,
we must still prove piecewise affinity everywhere.
In this case, the integrality of (b) is not burdened with an extra
denominator of $p^n$, so we may repeat the argument from 
\cite[Lemma~11.6.3]{kedlaya-course}; see Step 3 of
Theorem~\ref{T:1-dim-var} for essentially the same argument.
\end{proof}

\begin{example}
When $j \in J$, we do not expect an integrality result as in the 
$j = 0$ case; see Remark~\ref{R:partial_j-no-integrality}.  
One can easily generate an example in which the strong integrality statement 
for $\partial_j$ fails, as follows.
Suppose $p>0$, $\alpha \in (p^{-1/(p-1)}, 1)$, and $|u_j| = 1$.  We take the rank one $\partial_j$-differential module $M$ over $K \langle \alpha / t, t \rangle$ generated by $\bbv$ with $\partial_j (\bbv) = t^{-1} \bbv$.  Thus, $f_1^{(j)}(M, r) = r$ for $r \in [0, -\log \alpha]$.  
By Corollary~\ref{C:pullback radius},
$f_1^{(j)}(\varphi^{(\partial_j)*}M, r) = \frac rp$.
\end{example}

\begin{remark} \label{R:no subharmonicity}
Besides the weakening of the integrality condition, there are some other
aspects in which Theorem~\ref{T:subsidiary-partial-j} is weaker
than its counterpart \cite[Theorem~11.3.2]{kedlaya-course} if $p>0$.
For one, the latter includes a subharmonicity assertion, which refers to
the algebraic closure of the residue field of $K$. It is awkward to 
add a subharmonicity assertion here because the residue field of $K$
is crucially imperfect, so that it can admit a nontrivial $p$-basis.
(By contrast, if $p=0$, we can achieve a subharmonicity result;
see Theorem~\ref{T:subharmonicity}.)
For another, Theorem~\ref{T:subsidiary-partial-j}(a) does not apply
in a neighborhood of a point $r_0$ at which $f_i^{(j)}(M, r_0) = -\log |u_j|$.
The argument in \cite[Lemma~11.6.3]{kedlaya-course} does not extend to this case 
because the weak integrality result does not give a lower bound on slopes.
On the other hand, we do not have a counterexample against the claim that
$f_i^{(j)}(M,r)$ is everywhere piecewise affine.
\end{remark}

\subsection{Decomposition by subsidiary radii}
\label{S:decomposition}

In this subsection, we prove some decomposition theorems over
annuli and discs, as in \cite[Chapter~12]{kedlaya-course}.
We start by a technical lemma, copied from \cite[Lemma~1.2.7]{kedlaya-swan2}.

\begin{lemma} \label{L:proj-intersect0}
Let 
\[
\xymatrix{
R \ar[r] \ar[d] &  S \ar[d] \\
T \ar[r] & U
}
\]
be a commuting diagram of inclusions of integral domains,
such that the intersection
$S \cap T$ within $U$ is equal to $R$. Let $M$ be a finite
locally free $R$-module. Then the intersection of $M \otimes_R S$
and $M \otimes_R T$ within $M \otimes_R U$ is equal to $M$.
\end{lemma}
\begin{proof}
Choose $\bbe_1, \dots, \bbe_n \in M$ which form a basis of $M \otimes_R
(\Frac R)$; then there exists $f \in R$ such that $fM \subseteq
R\bbe_1 + \cdots + R\bbe_n$. Given $\bbv \in M \otimes_R U$
which belongs to both $M \otimes_R S$ and $M \otimes_R T$,
we can uniquely write $f\bbv = c_1 \bbe_1 + \cdots + c_n \bbe_n$
with $c_i \in U$. From the intersection property, we have
$c_i \in R$ for $i=1, \dots, n$, whence $f\bbv \in M$.

Since $M$ is locally free, as we vary the basis $\bbe_1, \dots, \bbe_n$, the
values of $f$ obtained generate
the unit ideal of $R$. We thus have $\bbv \in M$, as desired.
\end{proof}

\begin{lemma} \label{L:proj-intersect}
Retain notation as in Lemma~\ref{L:proj-intersect0}. Then any direct
sum decompositions of $M \otimes_R S$ and $M \otimes_R T$ which agree
on $M \otimes_R U$ are induced by a unique direct sum decomposition of $M$.
\end{lemma}
\begin{proof}
Apply Lemma~\ref{L:proj-intersect} to the idempotents in $M^\dual \otimes M$
giving the projections onto the factors in the decompositions.
\end{proof}

\begin{lemma}\label{L:criterion-unit-1}
Given $\alpha < \beta$ and $x \in K \{\{ \alpha / t, t / \beta \}\}$
such that
the function $r \mapsto \log |x|_{e^{-r}}$ is affine for $r \in (-\log \beta, -\log \alpha)$, then $x$ is a unit in $K \{\{\alpha / t, t / \beta\}\}$.
\end{lemma}
\begin{proof}
The condition is equivalent to saying that the Newton polygon of $x$ does not 
have any slopes in $(-\log \beta, -\log \alpha)$.  
This immediately implies the claim.
\end{proof}

\begin{lemma}\label{L:slope-decomp-coprime}
Let $P = \sum_i P_i T^i$ and $Q = \sum_i Q_i T^i$ be polynomials over 
$K \langle \alpha / t, t / \beta \rangle$ satisfying the following conditions.
\begin{enumerate}
\item[(a)]
We have $|P-1|_\gamma < 1$ for all $\gamma \in [\alpha,\beta]$.
\item[(b)]
For $d = \deg(Q)$, $Q_d$ is a unit and
$|Q|_\gamma = |Q_d|_\gamma$ for all $\gamma \in [\alpha,\beta]$.
\end{enumerate}
Then $P$ and $Q$ generate the unit ideal in
$K \langle \alpha / t, t / \beta \rangle [T]$.
\end{lemma}
\begin{proof}
We may assume without loss of generality that $Q_d = 1$.  
The hypothesis that $|Q|_\gamma = |Q_d|_\gamma$ for all
$\gamma \in [\alpha, \beta]$ implies that 
if
$S$ is the remainder upon dividing $R$ by $Q$, then $|S|_\gamma \leq |R|_\gamma$ for all
$\gamma \in [\alpha, \beta]$ (compare \cite[Lemma~2.3.1]{kedlaya-course}). 
If we then let
$S_i$ denote the remainder upon dividing $(1-P)^i$ by $Q$, the series 
$\sum_{i=0}^\infty S_i$ converges in $K \langle \alpha/t, t/\beta \rangle[T]$
(since the degrees of the $S_i$ are bounded by $d-1$)
and its limit $S$ satisfies 
$PS \equiv 1 \pmod{Q}$.
\end{proof}

\begin{theorem}\label{T:j-decomp-open-annulus}
Fix $j \in J^+$.  Let $M$ be a $\partial_j$-differential module of rank $d$ 
on $A_K^1(\alpha,\beta)$.
Suppose that the following conditions hold for some $i \in \{1,\dots,d-1\}$.
\begin{enumerate}
\item[(a)]
 The function 
$F^{(j)}_i(M,r)$ is affine for $-\log \beta < r < -\log \alpha$.
\item[(b)]
We have $f^{(j)}_i(M,r) > f^{(j)}_{i+1}(M,r)$ for $-\log \beta < r < -\log \alpha$.
\end{enumerate}
Then $M$ admits a unique direct sum 
decomposition separating the first $i$ subsidiary $\partial_j$-radii of $M \otimes F_\eta$ for any $\eta \in (\alpha, \beta)$.
\end{theorem}
\begin{proof}
When $j = 0$, this is \cite[Theorem~12.4.2]{kedlaya-course}; we thus
assume hereafter that $j \in J$.
The proof is similar to those of \cite[Theorems~12.2.2~and~12.3.1]{kedlaya-course};
for the benefit of the reader, we fill in some of the key details.

By Lemma~\ref{L:proj-intersect}, 
we may enlarge $K$ as needed; in particular,
we may reduce to the case $|u_j| = 1$ 
as in the proof of Theorem~\ref{T:subsidiary-partial-j}.  
Since the decomposition is unique if it exists, it is sufficient
to exhibit it on an open cover of $(\alpha,\beta)$ and then glue.
That is, it suffices to work in a 
neighborhood of any fixed $\gamma \in (\alpha, \beta)$;
again, we may enlarge $K$ to reduce to the case $\gamma=1$.

Suppose first that $f_i^{(j)}(M, 0) > -\log \omega$.  
Set notation as in the proof of Theorem~\ref{T:subsidiary-partial-j}. 
For some sufficiently small $c> 1$, 
we can choose $\gamma_1 \in (\alpha,1)$ and $\gamma_2 \in (1,\beta)$
such that the coefficient of $T^{d-i}$ in the characteristic polynomial 
$Q(T)$ of $N_c$ computes $F_i^{(j)}(M,r)$ for $r \in [-\log \gamma_2, -\log
\gamma_1]$; by (a), we may apply
Lemma~\ref{L:criterion-unit-1} 
(after changing $\gamma_1, \gamma_2$ slightly)
to deduce
that this coeffficient is a unit in 
$K \langle \gamma_1 / t, t / \gamma_2 \rangle$. 
By (b), we can apply \cite[Theorem~2.2.2]{kedlaya-course}
to factor $Q = Q_2 Q_1$ so that the roots of $Q_1$ are the $i$ largest roots of $Q$ under $|\cdot|_\gamma$ for all $\gamma \in [\gamma_1, \gamma_2]$.  (This is true for all $\gamma$ simultaneously because the construction is purely algebraic and \cite[Theorem~2.2.2]{kedlaya-course} takes care of convergence of the procedure.)

Use the basis $B_c$ to identify $M$ with 
$K \langle \gamma_1 / t, t / \gamma_2 \rangle^d$. 
Then we obtain a short exact sequence
$$
\exact 0 {\Ker(Q_1(N_c))} M {\Coker(Q_2(N_c))} 0
$$
of free modules over $K \langle \gamma_1 / t, t / \gamma_2 \rangle$.  
(The quotient is free because by Lemma~\ref{L:slope-decomp-coprime} 
applied after rescaling, 
$Q_1$ and $Q_2$ generate the unit ideal in 
$K \langle \gamma_1/t, t/\gamma_2 \rangle[T]$.)
Applying Lemma~\ref{L:lattice} to both factors 
(again for $c>1$ sufficiently small, and a choice of
$\gamma_1, \gamma_2$ depending on $c$), 
we construct a basis of $M$ on which $\partial_j$ acts via a matrix
\[
N'_c = \mat {A_c} {B_c} {C_c} {D_c}
\]
for which the following conditions hold.
\begin{enumerate}
\item[(a)]
The matrix $A_c$ is invertible and $|A_c^{-1}|_\gamma \cdot \max\{|\partial_j|_\gamma,
|B_c|_\gamma, |C_c|_\gamma, |D_c|_\gamma \} < 1$ for all $\gamma \in [\gamma_1, \gamma_2]$.
\item[(b)]
The Newton slopes of $A_c$ under $|\cdot|_\gamma$
account for the first $i$ subsidiary radii of $M \otimes F_\gamma$ for 
all $\gamma \in [\gamma_1, \gamma_2]$.
\end{enumerate}

By \cite[Lemma~6.7.1]{kedlaya-course}, $M$ admits a differential submodule accounting for the last $n-i$ subsidiary radii of $M \otimes F_\gamma$ for all $\gamma \in [\gamma_1, \gamma_2]$.  By repeating this argument for $M^\dual$, we obtain the desired splitting.

To deduce the theorem in the case $p>0$
without assuming that $f_i^{(j)}(M, 0) > \frac{1}{p-1}
\log p$, we prove the theorem in the case when 
$f_i^{(j)}(M, 0) > \frac 1{p^n(p-1)} \log p$ by induction on $n$, 
using $\partial_j$-Frobenius pushforward. This is sufficient because (b)
forces $f_i^{(j)}(M,0) > 0$, so there exists some $n$ for which
$f_i^{(j)}(M, 0) > \frac 1{p^n(p-1)} \log p$.
\end{proof}

\begin{caution}
In Theorem~\ref{T:j-decomp-open-annulus},
$M$ is only a locally free coherent sheaf and need not be
free, because the annulus on which we are working is not closed. 
Even if $M$ is free, the summands need not be free unless $K$
is spherically complete, in which case any locally free coherent sheaf
on $A_K^1(\alpha,\beta)$ is free.
\end{caution}

\begin{remark} \label{R:no-closed-decomp}
In \cite[Chapter~12]{kedlaya-course}, the analogous development starts with
a full decomposition theorem over a closed annulus
\cite[Theorem~12.2.2]{kedlaya-course}. We cannot do this 
here because we have not established an
analogue of subharmonicity \cite[Theorem~11.3.2(c)]{kedlaya-course} for 
$\partial_j$-differential modules,
except in the case $p=0$ (see Theorems~\ref{T:decomp-annulus-partial-j}
and~\ref{T:decomp-disc-partial-j}).
We can however recover partial
decomposition theorems over a closed disc or annulus,
analogous to \cite[Theorems~12.5.1 and~12.5.2]{kedlaya-course},
as follows.
\end{remark}

\begin{lemma}\label{L:criterion-unit-2}
\begin{enumerate}
\item[(a)]
For $x \in K \llbracket t \rrbracket_0$ nonzero, 
$x$ is a unit if and only if $|x|_{e^{-r}}$ is constant in a neighborhood of 
$r=0$.
\item[(b)]
For $x \in \cup_{\alpha \in (0,1)} K \langle \alpha/t, t \rrbracket_0$ 
nonzero, $x$ is a unit if and only if the function $r \mapsto \log |x|_{e^{-r}}$ is affine in some neighborhood of $0$.
\end{enumerate}
\end{lemma}
\begin{proof}
We may assume that $|x|_1 = 1$.  For (a), this means that
$x \in \gotho_K \llbracket t \rrbracket$.  Hence, $x = \sum_{i = 0}^\infty a_i t^i$ is a unit if and only if $a_0$ is a unit in $\gotho_K$, which is equivalent to $|x|_{e^{-r}}$ being constant in a neighborhood of $r = 0$.  For (b), by \cite[Lemma~8.2.6(c)]{kedlaya-course}, $x$ is a unit if and only if its image modulo 
$\gothm_K$ in $k((t))$ is a unit or equivalently nonzero, which is 
equivalent to the function $r \mapsto \log |x|_{e^{-r}}$ being affine in some neighborhood of $0$.
\end{proof}

\begin{theorem} \label{T:part-decomp-annulus}
Fix $j \in J^+$.  Let $M$ be a $\partial_j$-differential module of rank $d$ over $A_K^1(\alpha, \beta]$.  Suppose that the following conditions hold for some $i \in \{1,\dots,d-1\}$.
\begin{enumerate}
\item[(a)]
The function 
$F^{(j)}_i(M,r)$ is affine for $-\log \beta \leq r < -\log \alpha$.
\item[(b)]
We have $f^{(j)}_i(M,r) > f^{(j)}_{i+1}(M,r)$ for $- \log \beta \leq r < -\log \alpha$.
\end{enumerate}
Then for any $\gamma \in (\alpha,\beta)$, $M \otimes K \langle \gamma/t, t/\beta \rrbracket_0$
admits a direct sum decomposition separating the first $i$ subsidiary $\partial_j$-radii of $M \otimes F_\eta$ for $\eta \in [\gamma, \beta)$.
\end{theorem}
\begin{proof}
We first obtain a decomposition of $M \otimes K \langle \delta/t, t/\beta
\rrbracket_0$ for some uncontrolled $\delta \in (\alpha,\beta)$,
by arguing as in Theorem~\ref{T:j-decomp-open-annulus}, but using Lemma~\ref{L:criterion-unit-2}(b) instead of Lemma~\ref{L:criterion-unit-1}. (So far we
have not used condition (a).) To get the desired result, it 
suffices to do so for $\gamma \in (\alpha, \delta)$. For this,
we use the fact that the 
decomposition of $M$ over $A_K^1(\alpha,\beta)$ given by
Theorem~\ref{T:j-decomp-open-annulus} is unique, so we may thus glue together
the decomposition of $M \otimes K \langle \delta/t, t/\beta \rrbracket_0$
with the decomposition from Theorem~\ref{T:j-decomp-open-annulus}.
More explicitly, this involves applying Lemma~\ref{L:proj-intersect}
to the following situation: for any $\epsilon \in (\delta, \beta)$,
we have
\[
K \langle \gamma/t, t/\epsilon \rangle
\cap 
K \langle \delta/t, t/\beta \rrbracket_0 =
K \langle \gamma/t, t/\beta \rrbracket_0
\]
within $K \langle \delta/t, t/\epsilon \rangle$.
\end{proof}

\begin{theorem} \label{T:part-decomp-disc}
Fix $j \in J^+$.  Let $M$ be a $\partial_j$-differential module of rank $d$ over 
$K \langle t/\beta \rangle$.
Suppose that the following conditions hold for some $i \in \{1,\dots,d-1\}$.
\begin{enumerate}
\item[(a)]
The function $F^{(j)}_i(M,r)$ is constant in a neighborhood of $r= -\log \beta$.
\item[(b)]
We have $f^{(j)}_i(M, -\log \beta) > f^{(j)}_{i+1}(M, -\log \beta)$.
\end{enumerate}
Then $M \otimes K \llbracket t/\beta \rrbracket_0$ 
admits a direct sum decomposition separating the first $i$ subsidiary $\partial_j$-radii of $M \otimes F_\eta$ for $\eta \in (0,\beta)$.
\end{theorem}
\begin{proof}
Similar to Theorem~\ref{T:j-decomp-open-annulus}, but using Lemma~\ref{L:criterion-unit-2}(a) instead of Lemma~\ref{L:criterion-unit-1}.
\end{proof}

\begin{remark} \label{R:bounded-in-discrete-case}
In Theorems~\ref{T:part-decomp-annulus} and~\ref{T:part-decomp-disc},
if $K$ is discrete and $\beta \in |K^\times|^{\QQ}$, 
we can begin with free differential modules
over the rings $K \langle \alpha/t, t/\beta \rrbracket_0$
and $K \llbracket t/\beta \rrbracket_0$, respectively.
(The main reason for the restrictive hypotheses is to ensure that
is that the resulting rings are noetherian; among other reasons, this
is needed to ensure that we may freely pass between
finite projective modules and finite locally free modules.)
Note that this requires extending
the definition of $f_i^{(j)}(M,r)$ to $r = -\log \beta$,
using the completion of $\Frac K \llbracket t/\beta \rrbracket_0$
for the $\beta$-Gauss norm instead of $F_\beta$.
(Compare \cite[Remark~12.5.4]{kedlaya-course}.)
\end{remark}

\subsection{Variation for multiple derivations}

\label{S:var-intri-gen-radii}
In this subsection, we study the variation of intrinsic generic radii 
of a differential module over a disc or annulus.
The results here more closely match those of \cite{kedlaya-course} than
in the case of a $\partial_j$-differential module with $j \in J$.

We first introduce a rotation construction, in the manner of
\cite{kedlaya-swan1}.
\begin{notation}\label{N:generic-rotation}
Fix $\eta_+ \in \RR^+$.  Assume that $|u_J| = 1$.  
Denote $\widetilde K$ to be the completion of $K (x_J)$ with 
respect to the $(\eta_+^{-1}, \dots, \eta_+^{-1})$-Gauss norm; view $\widetilde K$ as a differential field of order $m$ with derivations $\serie{\partial_}m$.  We may use Taylor series (as in Lemma~\ref{L:generic-point})
to define, for any $\eta_- \in [0, \eta_+)$,
an injective homomorphism 
$\tilde f^*: K \langle \eta_- / t, t/\eta_+ \}\} 
\rar \widetilde K
\langle \eta_- / t, t/\eta_+ \}\}$
such that $\tilde f^* (u_j) = u_j + x_j t$.

For $\eta \in [0, \eta_+)$, we use 
$\widetilde F_\eta$ to denote the completion of 
$\widetilde K(t)$ with respect to the 
$\eta$-Gauss norm.  
Then $\tilde f^*$ extends to an injective isometric homomorphism $\tilde f^*: F_\eta \inj \widetilde F_\eta$.
\end{notation}

\begin{lemma}\label{L:generic-rotation}
For any subinterval $I$ of $[0, \eta_+)$ and
any $\partial_{J^+}$-differential module $M$ on $A_K^1 (I)$, $\tilde f^* M$ gives a $\partial_0$-differential module on 
$A_{\widetilde K}^1 (I)$.
Moreover, for $\eta \in I$,
$$
R_{\partial_0}(M \otimes \widetilde F_\eta) = \min \left\{ \eta IR_{\partial_0}(M \otimes F_\eta);\ \eta_+
IR_{\partial_j}(M \otimes F_\eta) \quad (j \in J) \right\}.
$$
\end{lemma}
\begin{proof}
This follows from the fact that
\[
\partial_0|_{\tilde f^* M}  =  \partial_0|_M + \sum_{j \in J} x_j
\partial_j|_M,
\]
after accounting for the different normalizations.
\end{proof}

\begin{notation} \label{N:subsidiary-radii2}
Let $M$ be a $\partial_{J^+}$-differential module of rank $d$ 
on $K \langle \alpha/t, t/\beta \rangle$.  
For $r \in [-\log \beta, -\log \alpha]$ and $i \in \{\serie{}d\}$, denote 
$$
f_i(M,r) = -\log IR(M \otimes F_{e^{-r}}; i), \quad F_i(M,r) = f_1(M,r) + \cdots + f_i(M,r).
$$
Note that we have changed the normalization from
Notation~\ref{N:subsidiary-radii-and-sums}, as we are now using intrinsic
rather than extrinsic radii.
\end{notation}

\begin{theorem}\label{T:1-dim-var}
Let $M$ be a $\partial_{J^+}$-differential module of rank $d$ on $A_K^1[\alpha, \beta]$. 
\begin{enumerate}
\item[(a)]
(Linearity)
For $i=\serie{}d$, the functions $f_i(M,r)$ and $F_i(M,r)$ are 
continuous and piecewise affine.
\item[(b)]
(Integrality)
If $i=d$ or $f_i(M,r_0) > f_{i+1}(M,r_0)$, then 
the slopes of $F_i(M,r)$ in some neighborhood of $r_0$ belong to $\ZZ$. Consequently, the slopes of each $f_i(M,r)$ and $F_i(M,r)$ belong to
$\frac{1}{1} \ZZ \cup \cdots \cup \frac{1}{d} \ZZ$.
\item[(c)]
(Monotonicity)
Suppose that $\alpha = 0$.
Then the slopes of $F_i(M,r)$
are nonpositive, and each $F_i(M,r)$ is constant for $r$ sufficiently large.
\item[(d)]
(Convexity)
For $i=1, \dots, d$, the function $F_i(M,r)$ is convex.
\end{enumerate}
\end{theorem}
\begin{proof}
Before proceeding, we reduce to the case $|u_J| = 1$ 
as in the proof of Theorem~\ref{T:subsidiary-partial-j}.
(Note that when enlarging $K$,
we do not retain the derivations with respect to any
added parameters.)

\textbf{\underline{Step 1:}} 
In this step, we prove that 
for $i = \serie{}d$, $f_i(M, r)$ and $F_i(M, r)$ are 
continuous at $r = -\log \beta$.  
Moreover, if $f_i(M, -\log \beta) > 0$, 
we show that there exists $\gamma \in [\alpha, \beta)$ such that 
(a) and (b) hold for $r \in [-\log \beta, -\log \gamma]$.
As in the proof of Theorem~\ref{T:subsidiary-partial-j},
we may reduce to the case $\beta=1$.

Let $R$ denote the completion of $\gotho_K ((t)) \otimes_{\gotho_K} K$ 
for the $1$-Gauss
norm; note that this contains both $F_1$ and
$K \langle \gamma/t, t \rrbracket_0$ for any $\gamma \in [\alpha,1)$.
We first apply 
Theorem~\ref{T:subsidiary-partial-0} (if $j=0$) or
Theorem~\ref{T:subsidiary-partial-j} (if $j \in J$), and
Theorem~\ref{T:part-decomp-annulus}, to decompose 
\[
M \otimes K \langle \gamma / t, t \rrbracket_0
= \bigoplus_{\lambda = 1}^{d'} M_\lambda^{[\gamma, 1]}
\]
for some $\gamma \in [\alpha, 1)$, in such a manner that
the following conditions hold for $j \in J^+$ and $\lambda = \serie{}d'$.
\begin{itemize}
\item[(i)]
The module $M_\lambda^{[\gamma, 1]} \otimes R$
is the base extension to $R$ 
of a differential submodule $M'_\lambda$ of $M \otimes F_1$ of pure intrinsic
$\partial_j$-radii.
\item[(ii)]
For $\mu = \serie{} \rank M_\lambda^{[\gamma, 1]}$ 
the function $f_\mu^{(j)}(M_\lambda^{[\gamma, 1]}, r)$
tends to $-\log IR_{\partial_j}(M'_\lambda)$ as $r \to 0^+$.
If $j=0$ or $IR_{\partial_j}(M'_\lambda) < 1$, then
also $f_\mu^{(j)}(M_\lambda^{[\gamma, 1]}, r)$
is affine for $r \in (0, -\log \gamma]$.
\end{itemize}
This alone suffices to imply continuity of $f_i(M,r)$ and $F_i(M,r)$
at $r = 0$.

Applying Theorem~\ref{T:j-decomp-open-annulus} 
after possibly making $\gamma$ closer to $1$, 
we get a further
decomposition $M_\lambda^{[\gamma, 1]} = \bigoplus_{\mu = 1}^{d_\lambda} M_{\lambda, \mu}^{[\gamma, 1)}$ over $A_K^1[\gamma, 1)$ such that
the following conditions hold for $\lambda = \serie{}d'$.
\begin{itemize}
\item [(iii)] 
For $j \in J^+$, $\mu = \serie{}d_\lambda$, 
if $IR_{\partial_j}(M'_\lambda) < 1$, 
then $M_{\lambda, \mu}^{[\gamma, 1)} \otimes F_{e^{-r}}$ has
pure intrinsic $\partial_j$-radii for $r \in (0, -\log \gamma]$.
\item [(iv)]
If $IR(M'_\lambda) < 1$, then for $j \in J^+$, $\mu = \serie{}d_\lambda$,
$\partial_j$ is dominant for $M_{\lambda, \mu}^{[\gamma, 1)} 
\otimes F_{e^{-r}}$ for some $r \in (0, -\log \gamma]$ 
if and only if the same holds for all $r \in (0, -\log \gamma]$.
\item[(v)]
If $\lambda, \lambda' \in \{\serie{}d'\}$ satisfy
$IR(M'_\lambda) > IR(M'_{\lambda'})$,
then $IR(M_{\lambda, \mu}^{[\gamma, 1)} 
\otimes F_{e^{-r}}) > IR(M_{\lambda', \mu'}^{[\gamma, 1)}, F_{e^{-r}})$ 
for all $\mu \in \{\serie{}d_\lambda \}$, 
$\mu' \in \{\serie{}d_{\lambda'} \}$ and $r \in (0, -\log \gamma]$.
\end{itemize}

The piecewise affinity from (a) in the case $f_i(M, 0) >0$ now
follows from Theorems~\ref{T:subsidiary-partial-0}(a)  
and \ref{T:subsidiary-partial-j}(a) applied to each
$M_{\lambda,\mu}^{[\gamma,1)}$.

To check (b), it suffices to verify integrality of slope
times rank for each component 
$M_{\lambda, \mu}^{[\gamma, 1)}$ 
for which $IR(M'_\lambda) < 1$.
If $\partial_0$ is dominant for $M_{\lambda, \mu}^{[\gamma, 1)} 
\otimes F_{e^{-r}}$ for some (hence all) $r \in (0, -\log \gamma]$, 
(b) follows from Theorem~\ref{T:subsidiary-partial-0}(b).  
Otherwise, pick arbitrary $\eta_-< \eta_+ \in (\gamma, 1)$ such that for $\eta \in (\eta_-, \eta_+)$,
$$
\eta_- / \eta_+ > 
IR(M_{\lambda, \mu}^{[\gamma, 1)} \otimes F_\eta) 
\big/ IR_{\partial_0}(M_{\lambda, \mu}^{[\gamma, 1)} \otimes F_\eta).
$$
Define $\widetilde K$ as in Notation~\ref{N:generic-rotation}.  
By Lemma~\ref{L:generic-rotation},
for $\eta \in (\eta_-, \eta_+)$, we have
\begin{eqnarray*}
R_{\partial_0}(\tilde f^* M_{\lambda, \mu}^{[\gamma, 1)} \otimes \widetilde F_\eta) &=& 
\min \left\{ \eta IR_{\partial_0}(M_{\lambda, \mu}^{[\gamma,1)} 
\otimes F_\eta);\ \eta_+ IR_{\partial_j}(M_{\lambda,\mu}^{[\gamma,1)} \otimes F_\eta) \quad (j \in J) \right\} \\
&=& \eta_+ IR(M_{\lambda,\mu}^{[\gamma,1)} \otimes F_\eta).
\end{eqnarray*}
In particular, $(f_1^{(0)})' (\tilde f^* M_{\lambda,\mu}^{[\gamma,1)}
\otimes \widetilde F_\eta, -\log \eta) = 
f'_1(M_{\lambda,\mu}^{[\gamma,1)}, -\log \eta) = 
(f_1)'_-(M_{\lambda,\mu}^{[\gamma,1)}, 0)$ for $\eta \in (\eta_-, \eta_+)$.
(Note that we showed in the proof of (a) that
$f_1(M_{\lambda,\mu}^{[\gamma,1)}, r)$ extends continuously
to $r=0$, so its left derivative at 0 makes sense.)
Thus, the statement (b) follows by applying  Theorem~\ref{T:subsidiary-partial-0}(b) to $\tilde f^* M_{\lambda, \mu}^{[\gamma, 1)}$.

\vspace{5pt}
\textbf{\underline{Step 1$'$:}} As a corollary of step 1, 
we deduce that for any $r_0 \in [-\log \beta, -\log \alpha]$,
$f_i(M,r)$ and $F_i(M,r)$ are continuous at $r_0$, and
in case $f_i(M,r_0) > 0$ one also has (a) and (b) in a neighborhood of $r_0$.
(In particular, we will then have continuity of $f_i(M,r)$ and $F_i(M,r)$
over all of $[-\log \beta, -\log \alpha]$.)
To make this deduction, we first replace $\beta$ by $\gamma = e^{-r_0}$
in case $ r_0 < -\log \alpha$, to obtain all the desired assertions
in a right neighborhood of $r_0$. By pulling back along $t \mapsto t^{-1}$
and then repeating the argument, we obtain the desired assertions
in a left neighborhood of $r_0$.

\vspace{5pt}
\textbf{\underline{Step 2:}}
In this step,
we prove that (d) holds in a neighborhood of each $r_0 \in (-\log \beta, 
-\log \alpha)$ for which $f_i(M,r_0) > 0$. 
It suffices to check in the case $f_i(M,r_0) > f_{i+1}(M,r_0)$, as the
general case follows by interpolation.

At this point, we may reduce to the case $r_0 = 0$.
As in Step 1, 
for some $\eta_- \in (\alpha, \eta)$,
we have a partial decomposition of $M$ over 
$K \langle \eta_-/t, t \rrbracket_0$
as $M = \bigoplus_{\lambda_- = 1}^{d_-} M_{\lambda_-}^{[\eta_-, 1]}$ satisfying (i) and (ii). 
For some $\eta_+ \in (1, \beta)$,
we also have a partial decomposition
over $K \langle \eta_+^{-1}/t, t \rrbracket_0$
of the pullback of $M$ along $t \mapsto t^{-1}$ as
$M = \bigoplus_{\lambda_+ = 1}^{d_+} M_{\lambda_+}^{[1, \eta_+]}$
satisfying appropriate analogues of (i) and (ii).
By making $\eta_-$ and $\eta_+$ closer to $1$, we may guarantee that for each index
$\lambda_-$ (resp.\ $\lambda_+$)
for which the ratio
$IR (M'_{\lambda_-}) / IR_{\partial_0}(M'_{\lambda_-})$
(resp.\ $IR (M'_{\lambda_+}) / IR_{\partial_0}(M'_{\lambda_+})$)
is less than 1, this ratio is also less than $\eta_- / \eta_+$.

Use Notation~\ref{N:generic-rotation}; by Theorem~\ref{T:subsidiary-partial-0},
$F_i^{(0)}(\tilde f^* M, r)$ is convex at $r = 0$.  
In particular, $(F_i^{(0)})'_-(\tilde f^* M, 0) \leq 
(F_i^{(0)})'_+(\tilde f^* M, 0)$.  
It suffices to show that
\begin{align}\label{E:1-dim-var-1}
(F_i^{(0)})'_+(\tilde f^* M, 0) - \theta_i(M, 0) &\leq (F_i)'_+(M, 0) \\
\label{E:1-dim-var-2}
(F_i^{(0)})'_-(\tilde f^* M, 0) - \theta_i(M, 0) &\geq (F_i)'_-(M, 0),
\end{align}
where $\theta_i(M, 0)$ denotes the sum of the dimensions of the
constituents $N$ of $M \otimes F_1$ for which 
$\partial_0$ is dominant and $f_1(N,0) \geq f_i(M,0)$.

The proofs of \eqref{E:1-dim-var-1} and~\eqref{E:1-dim-var-2} are similar,
so we focus on \eqref{E:1-dim-var-1}.
Decompose $M$ as in Step 1.
For each $\lambda$ such that $\partial_0$ is dominant for $M'_\lambda$,
we have by Lemma~\ref{L:generic-rotation} that in a punctured right 
neighborhood of $r = 0$,
\[
F_1^{(0)}(\tilde f^* M_{\lambda,\mu}^{[\gamma,1)}, r) = 
F_1^{(0)}(M_{\lambda,\mu}^{[\gamma,1)}, r)
\]
and so
\[
(F_1^{(0)})'_+(\tilde f^* M_{\lambda,\mu}^{[\gamma,1)}, 0) - 1 = 
(F_1^{(0)})'_+(M_{\lambda, \mu}^{[\gamma,1)}, 0) - 1 
\leq (F_1)'_+(M_{\lambda, \mu}^{[\gamma,1)}, 0).
\]
(The term $-1$ comes from the
change of normalization from Notation~\ref{N:subsidiary-radii-and-sums}
to Notation~\ref{N:subsidiary-radii2}. The inequality can be strict if
$\partial_j$ is also dominant for $M'_\lambda$ for some $j > 0$.)
For each $\lambda$ such that $\partial_0$ is not dominant for $M'_\lambda$,
we have by Lemma~\ref{L:generic-rotation} (and the choice of $\eta_+, \eta_-$)
that in a punctured right neighborhood of $r = 0$,
\[
F_1^{(0)}(\tilde f^* M_{\lambda,\mu}^{[\gamma,1)}, r) = 
F_1^{(j)}(M_{\lambda,\mu}^{[\gamma,1)}, r) - \log \eta_+
\]
and so
$$
(F_1^{(0)})'_+(\tilde f^* M_{\lambda,\mu}^{[\gamma,1)}, 0)
= (F_1)'_+(M_{\lambda, \mu}^{[\gamma,1)}, 0).
$$
Summing over components yields \eqref{E:1-dim-var-1}.

\vspace{5pt}
\textbf{\underline{Step 3:}}
In this step, we prove (a), (b), (d) in general, by induction on $i$.
Keep in mind that we already have the continuity aspect of (a) in general
(by Step $1'$), and all of (a), (b), (d) in a neighborhood of any 
$r_0 \in [-\log \beta, -\log \alpha]$
for which $f_i(M,r_0) > 0$ (by Steps 1, $1'$, 2).

We first check the piecewise affinity aspect of (a) in a right
neighborhood of some $r_0$ for which $f_i(M,r_0) = 0$.
By the induction hypothesis, we can pick $r_1 > r_0$ such that
$F_{i-1}(M,r)$ is affine on $[r_0, r_1]$.
Suppose that $r_2 \in (r_0,r_1)$ is a value for which $f_i(M,r_2) > 0$. 
By continuity of $f_i$, there exists an open neighborhood of 
$r_2$ on which $f_i(M,r)$ is everywhere positive. Let $U$
be the union of all such neighborhoods in $[r_0, r_1]$; 
then $U$ is an open interval
$(r_3, r_4)$, and $f_i(M,r_3) = 0$. Since (a) and (d) hold in a neighborhood
of each $r \in U$, $F_i(M,r)$ and hence $f_i(M,r)$ are piecewise affine
and convex on $U$.
In order for $f_i(M,r)$ to both be convex and to tend to $0$ as $r \to r_3^+$, 
$f_i(M,r)$ must have no nonpositive slopes; that is, $f_i(M,r)$ is
strictly increasing on $U$. However, we must also have $ f_i(M,r_4) = 0$
unless $r_4 = r_1$. The former possibility leads to a contradiction, so
we must have $r_4 = r_1$.

To sum up the previous paragraph, we now know that if there exists
$r_2 \in (r_0, r_1]$ such that $f_i(M,r_2) > 0$, then
$f_i(M,r) > 0$ for all $r \in [r_2, r_1]$. Consequently, 
on some right neighborhood of $r_0$,
$f_i(M,r)$ is either everywhere zero or everywhere positive.
In the former case, $f_i(M,r)$ is clearly affine on a right neighborhood
of $r_0$.
In the latter case, pick $r_2 \in (r_0,r_1]$ for which $f_i(M,r_2) > 0$;
then the slopes of $f_i(M,r)$ on $(r_0, r_2]$ are nondecreasing,
bounded below by 0, and (by (b)) confined to a discrete subset of $\RR$.
Consequently, there must be a least slope achieved,
occurring on a right neighborhood of $r_0$.
We thus deduce (a) in a right neighborhood of $r_0$. By symmetry,
the same argument applies to left neighborhoods; we may thus deduce (a)
in general.

Since (a) is now known, $f_i(M,r)$ takes only finitely many slopes
on all of $[-\log \beta, -\log \alpha]$. Except possibly for the slope 0,
each slope must occur at some $r$ for which $f_i(M,r) > 0$;
consequently, the knowledge of (b) at such points now implies (b) in general.

Finally, we still need to check (d) in a neighborhood
of a point $r_0$ at which $f_i(M,r_0) = 0$. By (a), $f_i(M,r)$ is affine
on a right neighborhood of $r_0$ and on a left neighborhood of $r_0$;
since $f_i(M,r) \geq 0$ everywhere, the right slope of $f_i(M,r)$  at $r_0$
must be greater than
or equal to the left slope of $f_i(M,r)$ at $r_0$. 
Since the same is true of $F_{i-1}(M,r)$ by the induction hypothesis,
the same must also be true of $F_i(M,r)$.
This yields (d).

\vspace{5pt}
\textbf{\underline{Step 4:}}
In this step, we prove (c).
By Dwork's transfer theorem (see Proposition~\ref{P:generic-point}),
for any $\eta < R_{\partial_0}(M \otimes F_\beta)$, $M \otimes K \langle t/\eta \rangle$
admits a basis in the kernel of $\partial_0$. In other words,
$M\otimes K \langle t/\eta \rangle$ is isomorphic to the pullback
of a $(\partial_J)$-differential module over 
$K$. Consequently, $F_i(M,r)$ is constant for $r$ sufficiently large;
by (d), this implies that $F_i(M,r)$ has all slopes nonpositive.
\end{proof}

\begin{remark} \label{R:rational-breaks}
If $p=0$, then the assertion that $r^{\dim V_r} \in |K^\times|$ in
Theorem~\ref{T:multi-decomp-over-field} implies that
$d! F_i(M,r) \in \log |K^\times| + \ZZ r$.
If $p > 0$, then we only deduce that for $h$ a nonnegative integer,
\[
f_i(M,r) > \frac{p^{-h}}{p-1} \log p
\quad \Longrightarrow \quad
d! F_i(M,r) \in p^{-h} \log |K^\times| + \ZZ r.
\]
In either case, we may conclude that the values of $r$ at which
$F_i(M,r)$ changes slope must belong to $\QQ \log |K^{\times}|$.
\end{remark}

\subsection{Decomposition for multiple variations}
\label{S:decomp-mult-1-dim}

We now obtain decomposition theorems which allow for multiple derivations.
\begin{theorem}
\label{T:decomp-multi-1-dim-annulus}
Let $M$ be a $\partial_{J^+}$-differential module of rank $d$ on $A_K^1(\alpha, \beta)$.  Suppose that the following conditions hold for some $i \in \{\serie{}d-1\}$.
\begin{enumerate}
\item[(a)]
The function 
$F_i(M,r)$ is affine for $-\log \beta < r < -\log \alpha$.
\item[(b)]
We have $f_i(M,r) > f_{i+1}(M,r)$ for $-\log \beta < r < -\log \alpha$.
\end{enumerate}
Then $M$ admits a unique direct sum 
decomposition separating the first $i$ subsidiary radii of $M \otimes F_\eta$ for any $\eta \in (\alpha, \beta)$.
\end{theorem}
\begin{proof}
Before proceeding, we reduce to the case $|u_J| = 1$ as in the proof of Theorem~\ref{T:j-decomp-open-annulus}.  It suffices to prove the decomposition in a neighborhood of each $r_0 \in (-\log \beta, -\log \alpha)$.  Again, we may assume $r_0 = 0$.

We continue with Step 2 in the proof of Theorem~\ref{T:1-dim-var}.  
We may further impose the auxiliary condition that
\begin{equation}\label{E:aux-condition}
- \log (\eta_-) < f_i(M, 0) - f_{i+1}(M, 0).
\end{equation}
By \eqref{E:1-dim-var-1} and the symmetric result, we have
\begin{equation}\label{E:1-dim-decomp-1}
(F_i)'_-(M, 0) \leq (F_i^{(0)})'_-(\tilde f^* M , 0) - \theta_i(M,0) \leq
(F_i^{(0)})'_+(\tilde f^* M , 0) - \theta_i(M,0) \leq (F_i)'_+(M, 0);
\end{equation}
all the inequalities are forced to be equalities as $F_i(M, r)$ is affine in a neighborhood of $r=0$.  In particular, $F_i^{(0)}(\tilde f^*M, r)$ is affine when $r \in (-\log \eta_+, -\log \eta_-]$.  We would get the decomposition by Theorem~\ref{T:j-decomp-open-annulus} if we knew that 
$f^{(0)}_i(\tilde f^*M, r) > f^{(0)}_{i+1}(\tilde f^*M, r)$ for $r$ in a neighborhood of $r=0$.  Indeed, by our auxiliary condition \eqref{E:aux-condition} and Lemma~\ref{L:generic-rotation},
$$
f_i^{(0)}(\tilde f^*M, 0) > \log (\eta_-) + f_i(M, 0) > f_{i+1}(M, 0) \geq f_{i+1}^{(0)}(\tilde f^*M, 0).
$$
The theorem follows.
\end{proof}

\begin{theorem}
\label{T:decomp-multi-1-dim-disc}
Let $M$ be a $\partial_{J^+}$-differential module of rank $d$ on $A_K^1[0, \beta)$.  Suppose that the following conditions hold for some $i \in \{\serie{}d-1\}$.
\begin{enumerate}
\item[(a)]
The function 
$F_i(M,r)$ is affine for $r > -\log \beta$.
(This implies $F_i(M,r)$ is constant by Theorem~\ref{T:1-dim-var}(c).)
\item[(b)]
We have $f_i(M,r) > f_{i+1}(M,r)$ for all (some) $r > -\log \beta$.
\end{enumerate}
Then $M$ admits a unique direct sum 
decomposition separating the first $i$ subsidiary radii of $M \otimes F_\eta$ for any $\eta \in (0, \beta)$.
\end{theorem}
\begin{proof}
Before proceeding, we reduce to the case $|u_J|=1$ as in the proof of Theorem~\ref{T:j-decomp-open-annulus}.  As noted in Step 4 of the proof of Theorem~\ref{T:1-dim-var}, there exists some $\eta \in (0, \beta)$ such that
$M\otimes K \langle t/\eta \rangle$ is isomorphic to the pullback
of a $(\partial_J)$-differential module $M_0$ over 
$K$.  Consequently, we have the
desired decomposition of $M$ over $A_K^1[0, \eta]$ by pulling back the decomposition of $M_0$ in the sense of Theorem~\ref{T:decomp-over-field-complete}.  The theorem follows by applying 
Theorem~\ref{T:decomp-multi-1-dim-annulus}
to $A_K^1(\eta', \beta)$ for some $\eta' \in (0, \eta)$.
\end{proof}

\begin{remark}
We can sometimes verify the hypotheses of
Theorem~\ref{T:decomp-multi-1-dim-disc}
using monotonicity and convexity (Theorem~\ref{T:1-dim-var}(c) and (d)).
For example, if $F'_i(M, r_0) = 0$, 
then $F_i(M, r)$ is constant for $r \geq r_0$.  Moreover, if we also have $f_i(M, r_0) > f_{i+1}(M, r_0)$, then 
condition (b) holds for $r \geq r_0$.
\end{remark}

\begin{remark}
As in Remark~\ref{R:no-closed-decomp}, we cannot state a decomposition theorem
over a closed annulus without assuming $p=0$
(in which case see Theorems~\ref{T:decomp-annulus}
and~\ref{T:decomp-disc}). However, we do get partial decomposition
theorems analogous to Theorems~\ref{T:decomp-annulus-partial-j}
and~\ref{T:decomp-disc-partial-j}, as follows.
\end{remark}

\begin{theorem}
\label{T:decomp-multi-1-dim-annulus-bounded}
Let $M$ be a $\partial_{J^+}$-differential module of rank $d$ on 
$A_K^1(\alpha, \beta]$.  Suppose that the following conditions hold for some $i \in \{\serie{}d-1\}$.
\begin{enumerate}
\item[(a)]
The function 
$F_i(M,r)$ is affine for $-\log \beta \leq r < -\log \alpha$.
\item[(b)]
We have $f_i(M,r) > f_{i+1}(M,r)$ for $-\log \beta \leq r < -\log \alpha$.
\end{enumerate}
Then for any $\gamma \in (\alpha, \beta)$,
$M \otimes K \langle \gamma/t, t/\beta \rrbracket_0$ 
admits a unique direct sum 
decomposition separating the first $i$ subsidiary radii of $M \otimes F_\eta$ for any $\eta \in (\gamma, \beta)$.
\end{theorem}
\begin{proof}
The fact that this holds for a single $\gamma$, even without hypothesis (a),
is a corollary of Step~1 of the proof of Theorem~\ref{T:1-dim-var}.
The desired conclusion follows by combining this assertion with
Theorem~\ref{T:decomp-multi-1-dim-annulus}.
\end{proof}

\begin{theorem}
\label{T:decomp-multi-1-dim-disc-bounded}
Let $M$ be a $\partial_{J^+}$-differential module of rank $d$ on 
$A_K^1[0, \beta]$.  
Suppose that the following conditions hold for some $i \in \{\serie{}d-1\}$.
\begin{enumerate}
\item[(a)]
The function 
$F_i(M,r)$ is affine for $r \geq -\log \beta$.
\item[(b)]
We have $f_i(M,-\log \beta) > f_{i+1}(M,-\log \beta)$.
\end{enumerate}
Then $M \otimes K \llbracket t/\beta \rrbracket_0$ admits a unique direct sum 
decomposition separating the first $i$ subsidiary radii of 
$M \otimes F_\eta$ for any $\eta \in (0, \beta)$.
\end{theorem}
\begin{proof}
This follows by combining 
Theorems~\ref{T:decomp-multi-1-dim-disc} 
and~\ref{T:decomp-multi-1-dim-annulus-bounded}.
\end{proof}

\begin{remark} \label{R:bounded-in-discrete-case-mult1}
As in Remark~\ref{R:bounded-in-discrete-case}, if $K$ is discretely
valued and $\beta \in |K^\times|^{\QQ}$, we can admit
modules in Theorems~\ref{T:decomp-multi-1-dim-annulus-bounded}
and~\ref{T:decomp-multi-1-dim-disc-bounded} defined directly
over the corresponding rings of bounded functions, namely
$K \langle \alpha/t, t/\beta \rrbracket_0$
and $K \llbracket t/\beta \rrbracket_0$.
\end{remark}

\subsection{An application to Swan conductors}

As promised earlier (Remark~\ref{R:includes swan1}), we can use the
results of this section to extend the results of \cite{kedlaya-swan1}
by relaxing \cite[Hypothesis~2.1.3]{kedlaya-swan1} to the hypothesis
that $K$ is of rational type. As this is straightforward to do,
we merely summarize the outcome by stating and deducing
a result which includes \cite[Theorems~2.7.2 and~2.8.2]{kedlaya-swan1}.

\begin{theorem} \label{T:swan1}
Let $M$ be a differential module of rank $d$ on $A^1_K(\eta_0,1)$
for some $\eta_0 \in (0,1)$, such that $IR(M \otimes F_{\rho}) \to 1$
as $\rho \to 1^-$. (That is, $M$ is solvable at $1$.)
Then for some $\eta \in (0,1)$, there exist a decomposition
$M = M_1 \oplus \cdots \oplus M_r$ on $A^1_K(\eta,1)$
and nonnegative
rational numbers $b_1,\dots,b_r$ with $\sum_i b_i \cdot \rank(M_i) \in \ZZ$,
such that 
\[
IR(M_i \otimes F_\rho; j) = \rho^{b_i} \qquad
(i=1,\dots,r; \quad j =1,\dots,\rank(M_i)).
\]
\end{theorem}
\begin{proof}
By Theorem~\ref{T:1-dim-var},
for $l=1,\dots,d$, the function $d! F_i(M, r)$ on $(0, -\log \eta)$
is continuous, convex, and piecewise affine with integer slopes.
By hypothesis, $d! F_i(M,r) \to 0$ as $r \to 0^+$; because of this
and the fact that $d! F_i(M,r) \geq 0$ for all $r$, the slopes of
$F_i(M,r)$ are forced to be nonnegative. Hence there is a least such
slope, that is, $d! F_i(M,r)$ is linear in a right neighborhood of $r=0$.

We can thus choose $\eta$ so that $d! F_i(M,r)$ is linear on
$(0, -\log  \eta)$ for $i=1,\dots,d$. We
obtain the desired decomposition by Theorem~\ref{T:decomp-multi-1-dim-disc};
the integrality of $\sum_i b_i \cdot \rank(M_i)$ follows from the fact that
$F_{d}(M,r)$ has integral slopes, again by Theorem~\ref{T:1-dim-var}.
\end{proof}

\subsection{Subharmonicity for residual characteristic $0$}

When $m=0$, the functions $F_i(M,r)$ obey a certain
subharmonicity property \cite[Theorem~11.3.2]{kedlaya-course}.
When the residual characteristic $p$ is equal to 0, one can obtain a 
similar result even when $K$ carries derivations. (See
Remark~\ref{R:no subharmonicity} for discussion of the case $p>0$.)

\begin{hypo}
Throughout this subsection, we assume $p=0$.
\end{hypo}

\begin{defn}
For $\overline \mu \in (k^\alg)^\times$, let $\mu$ be a lift of $\overline \mu$ in some finite extension $L$ of $K$. Let $E$ be a finite extension of the completion of $\gotho_K[t]_{(t)} \otimes_{\gotho_K} L$
for the $1$-Gauss norm. For $\alpha \leq 1 \leq \beta$, define the substitution
\[
T_{\mu}: K \langle \alpha/t, t/\beta \rangle \to E, \qquad
t \mapsto t + \mu.
\]
\end{defn}

\begin{defn}
Fix $j \in J^+$.
Let $M$ be a $\partial_j$-differential module of rank $d$ on 
$A_K^1[\alpha, \beta]$ for some $\alpha \leq 1 \leq \beta$.
For $i=1,\dots,n$, let $s_{\infty, i}^{(j)}(M)$ and $s_{0,i}^{(j)}(M)$ be the left (if $\beta \neq 1$) and right (if $\alpha \neq 1$) slopes of $F_i^{(j)}(M,r)$ at $r=0$.
For $\overline{\mu} \in (k^\alg)^\times$, 
pick any $\mu \in \gotho_L$ lifting $\overline{\mu}$ in a finite unramified extension $L$ of $K$,
and let $s^{(j)}_{\overline{\mu},i}(M)$ be the right slope of 
$F^{(j)}_i(T_\mu^*(M), r)$ at $r=0$.  Note that $T^*_\mu(M)$ is still a $\partial_j$-differential module by Lemma~\ref{L:H-under-unram-extn}.

If $M$ is a $\partial_{J^+}$-differential module of rank $d$ on 
$A_K^1[\alpha, \beta]$ for some $\alpha \leq 1 \leq \beta$, for $i=1,\dots,n$ and $\overline{\mu} \in k^\alg$, we similarly define $s_{\infty, i}(M)$ and $s_{\overline{\mu},i} (M)$ as the slopes of the corresponding functions $F_i(M, r)$ or $F_i(T_\mu^*(M), r)$.
\end{defn}

\begin{theorem} \label{T:subharmonicity-partial-j}
Fix $j \in J^+$.  Let $M$ be a $\partial_j$-differential module of rank $d$ on 
$A_K^1[\alpha, \beta]$ for some $\alpha < 1 < \beta$.
Choose $i \in \{1,\dots,d\}$ such that $f_i^{(j)}(M,0) > 0$.
\begin{enumerate}
\item[(a)]
The quantity $s_{\overline{\mu},i}^{(j)}(M)$ does not depend on the lift $\mu$ and the unramified extension $L/K$.
\item[(b)]
We have $s_{\overline{\mu},i}^{(j)}(M) \leq 0$ for all $\overline{\mu} \neq 0$,
with equality for all but finitely many $\overline{\mu}$.
\item[(c)]
We have
\[
s_{\infty, i}^{(j)}(M) \leq \sum_{\overline{\mu} \in k^\alg} s_{\overline{\mu},i}^{(j)}(M),
\]
with equality if either $i = n$ and $f^{(j)}_n(M, 0)>0$ or $i<n$ and $f^{(j)}_i(M, 0) \geq f^{(j)}_{i+1}(M, 0)$.
\end{enumerate}
\end{theorem}
\begin{proof}
When $j = 0$, this is \cite[Theorem~11.3.2(d)]{kedlaya-course}.  When $j \in J$, the proof of Theorem~\ref{T:subsidiary-partial-j} reduces the problem to \cite[Theorem~11.2.1(c)]{kedlaya-course}.  Note that we do not have to use the Frobenius push\-forward.
\end{proof}

\begin{remark}\label{R:s_mu-independent-of-BC}
Let $L$ be a complete extension of $K$ such that $\partial_j$ extends to $L$.  Then $M \otimes L$ becomes a $\partial_j$-differential module over $A_L^1[\alpha, \beta]$. For $\overline \mu \notin k^{\alg}$, we always have
$s_{\overline \mu, i}^{(j)}(M) = 0$; this can be seen either by inspecting
the proof of Theorem~\ref{T:subharmonicity-partial-j}, or by 
deducing the claim directly from (b). Namely, 
(b) implies that the equality $s^{(j)}_{\overline{\mu},i}(M) = 0$ holds with
only finitely many exceptions; on the other hand, if $\overline{\mu}$
were an exception not in $k^{\alg}$, then so would be each of its infinitely
many conjugates in an algebraic closure
of the residue field of $L$.
\end{remark}

\begin{theorem} \label{T:subharmonicity}
Let $M$ be a $\partial_{J^+}$-differential module of rank $d$ on 
$A_K^1[\alpha, \beta]$ for some $\alpha < 1 < \beta$.
Choose $i \in \{1,\dots,d\}$ such that $f_i(M,0) > 0$.
\begin{enumerate}
\item[(a)]
The quantity $s_{\overline{\mu},i}(M)$ does not depend on the lift $\mu$ and the unramified extension $L/K$.
\item[(b)]
We have $s_{\overline{\mu},i}(M) \leq 0$ for all $\overline{\mu} \neq 0$,
with equality for all but finitely many $\overline{\mu}$.
\item[(c)]
We have
\[
s_{\infty, i}(M) \leq \sum_{\overline{\mu} \in k^\alg} s_{\overline{\mu},i}(M).
\]
\end{enumerate}
\end{theorem}
\begin{proof}
Suppose first that $\partial_0$ is dominant for each irreducible component of
$M \otimes F_1$ which contributes to $F_i(M,0)$. Then
$s_{\infty,i}(M)$ is less than or equal to the left slope of
$F_i^{(0)}(M,r)$ at $r=0$,
whereas $s_{\overline{\mu},i}(M)$ is greater than or equal to the right
slope of $F_i^{(0)}(T_\mu^*(M),r)$ at $r=0$.
We may thus reduce to the case $m=0$, which is
\cite[Theorem~11.3.2(c)]{kedlaya-course}.

It suffices to reduce to the case where $\partial_0$ is dominant for
each irreducible
component of $M \otimes F_1$ which contributes to $F_i(M,0)$.
This proceeds as in Step~2 of the proof of
Theorem~\ref{T:1-dim-var}, except that we may end up working over an
enlargement of $K$. This causes no harm in (a) or (b), but in (c)
the sum may end up running over a larger field.
However, the argument of Remark~\ref{R:s_mu-independent-of-BC} 
shows that the extra terms do not contribute: that is, we may use (b)
to show that $s_{\overline{\mu},i}(M) = 0$ 
if $\overline{\mu} \notin k^{\alg}$, so (c) holds as written.
\end{proof}

\begin{remark}
The proof given above does not achieve the equality in (c) for $m>0$,
because the reduction in the last paragraph does not maintain
equality.
\end{remark}

As in \cite[Subsection~12.2]{kedlaya-course}, we can study decomposition theorems over closed annuli or discs using subharmonicity. 

\begin{defn}
Fix $j \in J^+$.  Let $M$ be a $\partial_j$-differential module over $K \langle \alpha/t, t/\beta \rangle$ with $\alpha \leq 1 \leq \beta$.  Define the \emph{$i$-th $\partial_j$-discrepancy} of $M$ at $r = 0$ as
\[
\disc^{(j)}_i(M, 0) = - \sum_{\overline \mu \in (k^\alg)^\times} s^{(j)}_{\overline \mu, i}(M);
\]
it is nonnegative by Theorem~\ref{T:subharmonicity-partial-j}. 
By Remark~\ref{R:s_mu-independent-of-BC}, this definition is invariant under
enlarging $K$.
 We may extend the definition to general $r \in [-\log \beta, -\log \alpha]$ by pulling back $M$ along 
\[
K \langle \alpha/t, t/\beta \rangle \rar K(c)^\wedge \langle \alpha e^r/t, t/\beta e^r \rangle, \quad t \mapsto ct,
\]
where $c$ is transcendental over $K$ and $K(c)^\wedge$ is the completion with respect to the $e^{-r}$-Gauss norm. 

If $M$ is a finite $\partial_{J^+}$-differential module over $K \langle \alpha/t, t/\beta \rangle$ with $\alpha \leq 1 \leq \beta$, we similarly define the \emph{$i$-th discrepancy} $\disc_i(M, 0)$ of $M$ at $r = 0$ as the sum of $-s_{\overline \mu, i}(M)$ over $\overline \mu \in (k^\alg)^\times$. 
This quantity is again nonnegative, and is again invariant under enlarging 
$K$ (this time by the final remark in the proof of 
Theorem~\ref{T:subharmonicity}).
This definition can similarly
be extended to $r \in [-\log \beta, -\log \alpha]$.
\end{defn}

\begin{remark}
If $r \notin \QQ \log |K^\times|$, then
Remark~\ref{R:rational-breaks} implies that $F_i(M,r)$ is affine in a 
neighborhood of $r$. By Theorem~\ref{T:subharmonicity},
it follows that $\disc_i(M,r) = 0$.
\end{remark}

\begin{theorem}\label{T:decomp-annulus-partial-j}
Fix $j \in J^+$.  Let $M$ be a $\partial_j$-differential module over $K \langle \alpha/ t, t / \beta \rangle$ of rank $d$. Suppose that the following conditions hold for some $i \in \{\serie{}d-1\}$. 
\begin{itemize}
\item[(a)] We have $f_i^{(j)}(M, r) > f_{i+1}^{(j)}(M, r)$ for $r \in [-\log \beta, -\log \alpha]$.
\item[(b)] The function $F_i^{(j)}(M, r)$ is affine for $r \in [-\log \beta, -\log \alpha]$.
\item[(c)] We have $\disc_i^{(j)}(M, -\log \alpha) = \disc_i^{(j)}(M, -\log \beta) = 0$.
\end{itemize}
Then there is a direct sum decomposition of M inducing, for each $\eta \in [\alpha, \beta]$, the decomposition of $M \otimes F_\eta$ separating the first $i$ subsidiary $\partial_j$-radii from the others.
\end{theorem}
\begin{proof}
Similar to Theorem~\ref{T:j-decomp-open-annulus} but invoking \cite[Lemma~12.1.3]{kedlaya-course} instead.
\end{proof}

\begin{theorem}\label{T:decomp-disc-partial-j}
Fix $j \in J^+$.  Let $M$ be a $\partial_j$-differential module over $K \langle t / \beta \rangle$ of rank $d$. Suppose that the following conditions hold for some $i \in \{\serie{}d-1\}$. 
\begin{itemize}
\item[(a)] We have $f_i^{(j)}(M, -\log \beta) > f_{i+1}^{(j)}(M, -\log \beta)$.
\item[(b)] The function $F_i^{(j)}(M, r)$ is constant for $r$ in a neighborhood of $-\log \beta$.
\item[(c)] We have $\disc^{(j)}_i(M, -\log \beta) = 0$.
\end{itemize}
Then there is a direct sum decomposition of M inducing, for each $\eta \in (0, \beta]$, the decomposition of $M \otimes F_\eta$ separating the first $i$ subsidiary $\partial_j$-radii from the others.
\end{theorem}
\begin{proof}
One can prove this similarly to Theorem~\ref{T:j-decomp-open-annulus} by invoking \cite[Lemma~12.1.2]{kedlaya-course} instead.  It is also an immediate corollary of Theorems~\ref{T:decomp-annulus-partial-j} and \ref{T:part-decomp-disc}; note that Theorem~\ref{T:subharmonicity-partial-j} verifies the condition (c) in Theorem~\ref{T:decomp-annulus-partial-j}.
\end{proof}

\begin{theorem}\label{T:decomp-annulus}
Let $M$ be a $\partial_{J^+}$-differential module over $K \langle \alpha/ t, t / \beta \rangle$ of rank $d$. Suppose that the following conditions hold for some $i \in \{\serie{}d-1\}$.
\begin{itemize}
\item[(a)] We have $f_i(M, r) > f_{i+1}(M, r)$ for $r \in [-\log \beta, -\log \alpha]$.
\item[(b)] The function $F_i(M, r)$ is affine for $r \in [-\log \beta, -\log \alpha]$.
\item[(c)] We have $\disc_i(M, -\log \alpha) = \disc_i(M, -\log \beta) = 0$.
\end{itemize}
Then there is a direct sum decomposition of M inducing, for each $\eta \in [\alpha, \beta]$, the decomposition of $M \otimes F_\eta$ separating the first $i$ subsidiary radii from the others.
\end{theorem}
\begin{proof}
Similar to Theorem~\ref{T:decomp-multi-1-dim-annulus} but invoking Theorem~\ref{T:decomp-annulus-partial-j} instead on the boundary.
\end{proof}

\begin{theorem}\label{T:decomp-disc}
Let $M$ be a $\partial_{J^+}$-differential module over $K \langle t / \beta \rangle$ of rank $d$. Suppose that the following conditions hold for some $i \in \{\serie{}d-1\}$. 
\begin{itemize}
\item[(a)] We have $f_i(M, -\log \beta) > f_{i+1}(M, -\log \beta)$.
\item[(b)] The function $F_i(M, r)$ is constant for $r$ in a neighborhood of $-\log \beta$.
\item[(c)] We have $\disc_i(M, -\log \beta) = 0$.
\end{itemize}
Then there is a direct sum decomposition of M inducing, for each $\eta \in (0, \beta]$, the decomposition of $M \otimes F_\eta$ separating the first $i$ subsidiary radii from the others.
\end{theorem}
\begin{proof}
It follows from Theorems~\ref{T:decomp-annulus} and \ref{T:part-decomp-disc}; note also that Theorem~\ref{T:subharmonicity} verifies the condition (c) in Theorem~\ref{T:decomp-annulus}.
\end{proof}

\section{Differential modules on higher-dimensional spaces}
\setcounter{equation}{0}

We now study the
variation of subsidiary radii of differential modules on some simple
higher-dimensional spaces. 
Rather than derive these directly, we deduce these from the corresponding
results on one-dimensional spaces from the previous section, using some
properties of convex functions.

Throughout this section, we retain Hypothesis~\ref{H:K}.

\subsection{Convex functions}

In this subsection, we set some terminology for convex functions,
as in \cite[Section~2]{kedlaya-part3}.

\begin{defn}
For a subset $C \subseteq \RR^n$, we denote its interior by $\inte(C)$.  We say it is \emph{convex} if for all $x,y \in C$ and all $t \in [0,1]$,
$tx + (1-t)y \in C$. For $C \subseteq \RR^n$ convex, a function $f: C \to 
\RR$ is \emph{convex}
if for all $x,y \in C$ and all $t \in [0,1]$,
\begin{equation} \label{eq:convexity def}
tf(x) + (1-t)f(y) \geq f(tx + (1-t)y).
\end{equation}
Such a function is continuous on $\inte(C)$.
\end{defn}

\begin{defn}
An \emph{affine functional} on $\RR^n$ is a map $\lambda: \RR^n \to \RR$ of the
form $\lambda(x_1,\dots,x_n) = a_1 x_1 + \cdots + a_n x_n + b$ for some
$a_1,\dots, a_n, b \in \RR$. If $a_1,\dots,a_n \in \ZZ$, we say $\lambda$
is \emph{transintegral} (short for ``integral after translation'');
if also $b \in \ZZ$, we say $\lambda$ is \emph{integral}.
For $\lambda: \RR^n \to \RR$ an affine functional, define the \emph{slope}
of $\lambda$ as the linear functional 
$\tilde{\lambda}(x) = \lambda(x) - \lambda(0)$.
\end{defn}

\begin{defn}
For $f: C \to \RR^n$ convex, a \emph{domain of affinity} of $f$ is a 
subset $U$ of $C$
with nonempty interior (in $\RR^n$) on which $f$ agrees 
with an affine functional $\lambda$.
The nonempty interior condition ensures that $\lambda$ is uniquely 
determined; we call it the \emph{ambient functional} on $U$.
\end{defn}

\begin{lemma} \label{L:slopes}
Let $f: C \to \RR^n$ be a convex function, and
let $\lambda: \RR^n \to \RR$ be an affine functional
which agrees with $f$ on a subset
of $C$ with nonempty interior in $\RR^n$.
\begin{enumerate}
\item[(a)]
We have $f(x) \geq \lambda(x)$ for all $x \in C$.
\item[(b)]
The set of $x \in C$ for which $f(x) = \lambda(x)$
is a convex subset of $C$.
\item[(c)]
If $\lambda'$ is another affine functional with the same slope as $\lambda$,
and $\lambda'$ occurs as the ambient functional of some domain of affinity
of $f$, then $\lambda = \lambda'$.
\end{enumerate}
\end{lemma}
\begin{proof}
For (a), choose $y$ in the interior of a domain of affinity $U$ of $f$
with ambient functional $\lambda$. For $\epsilon > 0$ sufficiently small,
the quantity $z$ defined by $\epsilon x + (1-\epsilon) z = y$ will also
belong to $U$. By convexity of $f$,
$\epsilon f(x) + (1-\epsilon) \lambda(z) \geq \lambda(y)$, so
\[
f(x) \geq \frac{\lambda(y) - (1-\epsilon) \lambda(z)}{\epsilon} = \lambda(x).
\]
We may deduce (b) and (c) immediately from (a).
\end{proof}

\begin{defn}
A subset $C \subseteq \RR^n$ is \emph{polyhedral} if there exist 
finitely many
affine functionals $\lambda_1, \dots, \lambda_r$ such that
\begin{equation}\label{E:TRP-set}
C = \{x \in \RR^n: \lambda_i(x) \geq 0 \qquad (i=1,\dots,r)\}.
\end{equation}
(We do not require $C$ to be bounded.)
If the $\lambda_i$ can all be taken to be (trans)integral, we say that $C$ is 
\emph{(trans)rational polyhedral}.
(We use \emph{RP} and \emph{TRP} 
as shorthand for \emph{rational polyhedral}
and \emph{transrational polyhedral}.)
For $C \subseteq \RR^n$ a convex subset of $\RR^n$,
a continuous convex function $f: C \to \RR^n$ is \emph{polyhedral} 
if there exist finitely many 
affine functionals 
$\lambda'_1, \dots, \lambda'_s$ such that
\begin{equation} \label{eq:polyhedral-fn}
f(x) = \max\{\lambda'_1(x), \dots, \lambda'_s(x)\} \qquad (x \in C).
\end{equation}
(In particular, such a function extends continuously to a convex
function on the closure of $C$, or even to all of $\RR^n$.)
Similarly, if $C$ is (trans)rational polyhedral, we
say $f$ is \emph{(trans)integral polyhedral} if 
\eqref{eq:polyhedral-fn} holds for some (trans)integral affine
functionals $\lambda'_1, \dots, \lambda'_s$.
\end{defn}

\begin{remark} \label{R:fin-many-affdoms}
If $C$ is a convex subset of $\RR^n$, then a continuous convex function
$f: C \to \RR^n$ is polyhedral if and only if $C$ is covered by 
finitely many domains of affinity
for $f$, by \cite[Lemma~2.2.6]{kedlaya-part3}.
Moreover, if $C$ is compact, then it suffices to check that 
every point in $C$ has a neighborhood covered by finitely many
domains of affinity for $f$, as then compactness will imply the
existence of finitely many domains of affinity which cover $C$.
\end{remark}

\subsection{Detecting polyhedral functions}

In this subsection, we establish a theorem that can be used to detect
polyhedrality of certain convex functions based on integrality properties
of certain values of the functions. We start with a weaker result
in the same spirit, from \cite[Section~2]{kedlaya-part3}.

\begin{notation}
In this subsection, for a point $x \in \QQ^n$, we write 
$\serie{x_}n$ for the coordinates of $x$.
\end{notation}

\begin{theorem} \label{T:integral-polyhedral}
Let $C$ be a bounded RP subset of $\RR^n$, and let
$f: C \to \RR$ be a continuous
convex function. Then $f$ is integral polyhedral if and only if 
\begin{equation} \label{eq:integral values}
f(x) \in \ZZ + \ZZ x_1 + \cdots + \ZZ x_n \qquad (x \in C \cap \QQ^n).
\end{equation}
\end{theorem}
\begin{proof}
See \cite[Theorem~2.4.2]{kedlaya-part3}.
\end{proof}

One cannot hope to similarly detect transintegral polyhedral functions
by sampling them at individual points, i.e., on zero-dimensional TRP subsets
of $\RR^n$. The best one can do is detect them by sampling on one-dimensional
TRP subsets of $\RR^n$, as follows.

\begin{defn}
Let $C$ be a convex subset of $\RR^n$. We say a function
$f: C \to \RR$ is \emph{convex transintegral polyhedral in dimension $1$}
if its restriction to the intersection of $C$ with any one-dimensional TRP
subset of $\RR^n$ is continuous, convex, and transintegral polyhedral.
In other words,
for any $x \in C, a \in \QQ^n$,
if  we put $I_{x,a} = \{t \in \RR: x+ta \in C \}$,
then the function $g: I_{x, a} \to \RR$ defined by
$g(t) = f(x+ta)$ is continuous, convex,
piecewise affine with slopes in $a_1 \ZZ + \cdots + a_n \ZZ$, 
and has only finitely many slopes. (The latter is automatic if
$I_{x,a}$ is closed and bounded, which always occurs if $C$ is compact.)
\end{defn}

\begin{theorem} \label{T:trans}
Let $C$ be a TRP subset of $\RR^n$.
Let $f: C \to \RR$ be a function which is convex transintegral polyhedral
in dimension $1$. Then $f$ itself 
is convex and transintegral polyhedral (hence continuous).
\end{theorem}
The proof is somewhat complicated, and will occupy the rest of this section.
We first tackle the case where $C$ is compact, for which
we assemble several lemmas.

\begin{defn}
Let $C$ be a TRP subset of $\RR^n$.
For $x \in C$,
define the \emph{angle} of $C$ at $x$, denoted $\angle_x C$, to be the set of
$z \in \RR^n$ such that for some $t_0 > 0$, 
$x + tz \in C$ for $t \in [0,t_0]$.  It is clear that $\angle_x C$ is an
RP subset of 
$\RR^n$ stable under multiplication by $\RR_{>0}$.
\end{defn}

\begin{lemma} \label{L:trp1 convex}
Let $C$ be a TRP subset of $\RR^n$, and let
$f: C \to \RR$ be a function which is convex transintegral polyhedral
in dimension $1$. Then $f$ is convex.
\end{lemma}
\begin{proof}
We may assume $\dim(C) = n$, by replacing
$\RR^n$ by a plane of the appropriate dimension.
It suffices to verify \eqref{eq:convexity def} for any $x,y \in C$
and any $t \in [0,1]$.
By applying a change of basis in $\GL_n(\ZZ)$,
we may reduce to the case where the standard basis vectors $\bbe_1,\dots,\bbe_n$
belong to $\angle_x C$. 

We now choose $x'_1,\dots,x'_n > 0$ in turn so that for $i=1,\dots,n$,
$x_i + x'_i - y_i \in \QQ$,
$x + x'_1 \bbe_1 + \cdots + x'_i \bbe_i \in \inte(C)$, and
\begin{gather*}
|f(x + x'_1 \bbe_1 + \cdots + x'_i \bbe_i) - 
f(x + x'_1 \bbe_1 + \cdots + x'_{i-1} \bbe_{i-1})| 
< \epsilon/n \\
|f(t(x + x'_1 \bbe_1 + \cdots + x'_i \bbe_i) + (1-t)y) - 
f(t(x + x'_1 \bbe_1 + \cdots + x'_{i-1} \bbe_{i-1}) + (1-t)y)| 
< \epsilon/n.
\end{gather*}
Namely, given $x'_1,\dots,x'_{i-1}$, 
the eligible choices
of $x'_i$ form a dense subset of an open interval with left endpoint $0$.
(Here we are using the continuity of the restriction of $f$ to TRP
sets of dimension 1.)

Put $x' = x + x'_1 \bbe_1 + \cdots + x'_n \bbe_n$. Since $x' - y \in \QQ^n$,
the segment from $x'$ to $y$ is T
RP. Hence
\[
tf(x') + (1-t) f(y) \geq f(tx' + (1-t)y),
\]
and so 
\[
tf(x) + (1-t)f(y) \geq f(tx + (1-t)y) - 2\epsilon.
\]
Since  $\epsilon$ was arbitrary, this implies 
\eqref{eq:convexity def}, yielding convexity of $f$.
\end{proof}

\begin{defn}
Let $C$ be a TRP subset of $\RR^n$.
For $f: C \to \RR$ a convex function, $x \in C$, and $z \in \angle_x C$, define
$f'(x, z)$ to be the directional derivative of $f$ at $x$
in the direction of $z$, i.e.,
\[
f'(x, z) = \lim_{t \to 0^+} \frac{f(x + tz) - f(x)}{t}.
\]
Note that this is a limit 
taken over a decreasing sequence; for it to exist in all 
cases, we must allow it to take the value $-\infty$.
\end{defn}

\begin{lemma} \label{L:dir deriv convex}
Let $C$ be a TRP subset of $\RR^n$,
and let $f: C \to \RR$ be a convex function.
For any fixed $x \in C$, the function $z \mapsto f'(x,z)$ is convex
as a function from $\angle_x C$ to $\RR \cup \{-\infty\}$
(in the sense of satisfying \eqref{eq:convexity def}).
\end{lemma}
\begin{proof}
Take any $z_1, z_2 \in \angle_x C$.  We assume first that $f'(x, z_1), f'(x, z_2) > -\infty$.
Pick $u \in [0,1]$ and put $z_3 = uz_1 + (1-u)z_2$.
Given $\epsilon > 0$, choose $t > 0$ for which 
\[
x + tz_i \in C \quad (i=1,2,3), \qquad
f'(x,z_i) \geq \frac{f(x+tz_i)-f(x)}{t} - \epsilon
\quad (i=1,2).
\]
Then
\begin{align*}
u f'(x,z_1) + (1-u) f'(x,z_2) & \geq
u \frac{f(x + tz_1) - f(x)}{t} + (1-u) 
\frac{f(x + tz_2) - f(x)}{t} - \epsilon \\
&\geq \frac{f\big(u(x+tz_1) + (1-u)(x+tz_2)\big) - f(x)}{t} - \epsilon \\
&\geq f'(x,z_3) - \epsilon.
\end{align*}
Since $\epsilon$ was arbitrary, this proves the claim when both $f'(x, z_1)$ and $f'(x, z_2)$ are not $-\infty$.  If one of them is $-\infty$, the same argument would imply that $f'(x, z_3) = -\infty$; this completes the proof.
\end{proof}

\begin{lemma} \label{L:dir deriv trp}
Assume that Theorem~\ref{T:trans} holds for compact 
$C$ with $n$ replaced by $n-1$.
Let $C$ be a compact TRP subset of $\RR^n$,
and let $f: C \to \RR$ be a function which is convex transintegral
polyhedral in dimension $1$. Then for any $x \in C$, the function
$z \mapsto f'(x,z)$ on $\angle_x C$ is itself convex transintegral polyhedral in
dimension $1$.
\end{lemma}
\begin{proof}
By Lemma~\ref{L:trp1 convex}, $f$ is convex.
By Lemma~\ref{L:dir deriv convex}, $f'(x,z)$ is convex on
$\angle_x C$, hence continuous on $\inte(\angle_x C)$.
By hypothesis, for $z \in \angle_x C \cap \QQ^n$,
$f'(x,z) \in \ZZ z_1 + \cdots + \ZZ z_n$.
By Theorem~\ref{T:integral-polyhedral}, 
$f'(x,z)$ is integral polyhedral on any bounded RP subset of
$\inte(\angle_x C)$.

By subdividing $C$ by hyperplanes, we may reduce to the case where
$\angle_x C$ admits a bounded cross-section by a rational hyperplane.
Pick any $z \in \angle_x C$ and $a \in \QQ^n$ such that the
set $I_{z, a} = \{u \in \RR: z + ua \in \angle_x C\}$ is bounded.
We must show that the function $g(u) = f'(x, z + ua)$
is continuous, convex, and transintegral polyhedral on $I_{z, a}$.
(This suffices because we can recover all values of $f'(x,z)$ from the values
on a bounded cross-section by a rational hyperplane.)
By what we know about $f'$, we already know all of these on
$\inte(I_{z, a})$. Consequently, it suffices to check that $g$ is affine
in a neighborhood of an endpoint of $I_{z, a}$.

For this, we may assume that the endpoint in question is a left
endpoint at $u=0$. Then $z$ lies on the boundary of $\angle_x C$,
so we can choose a codimension 1 facet $D$ of $C$ containing $x$,
such that the ray from $x$ in the direction of $z$ has nontrivial
intersection with $D$. By the hypothesis that Theorem~\ref{T:trans}
holds on compact TRP subsets of 
dimension $n-1$, the restriction of $f$ to $D$ must be 
transintegral polyhedral. In particular, we can rescale $z$ so that
for $t \in [0,1+\epsilon]$ for some $\epsilon > 0$, $x+tz \in C$
and $f(x+tz) = f(x) + t f'(x,z)$.

Consider the function $h(t) = f'(x+tz, a)$ for $t \in [0, 1+\epsilon]$.
Since the difference quotient $(f(x+tz+ua)-f(x+tz))/u$ is convex
in $t$ (the term $f(x+tz+ua)$ is convex, the term $-f(x+tz)$
is affine, and dividing by $u$ has no effect), so is $h(t)$. However,
$h(t) \in \ZZ a_1 + \cdots + \ZZ a_n$ for all $t$. This means 
that
for $t \in (0,1+\epsilon)$, 
$h(t)$ is continuous but takes values in a discrete subset
of $\RR$; this can only happen if $h(t)$
is equal to a constant value $c$ on $(0, 1+\epsilon)$.

Rescale $a$ if necessary so that 
$x + z + a \in C$ and 
$f(x + z + ua) = f(x) + f'(x,z) + uc$ for $u \in [0,1]$.
We now claim that $f(x + tz + ua) = f(x) + t f'(x,z) + uc$
for $t \in [0,1], u \in [0,t]$. 
Since equality holds at $(t,u) = (0,0), (1,0), (1,1)$, we have
by convexity of $f$ that $f(x + tz + ua) \leq f(x) + t f'(x,z) + uc$
in the entire region. On the other hand, for any $t \in [0,1]$,
the function $f(x+tz+ua)$ in $u$ is convex, and equals
$f(x) + t f'(x,z) + uc$ for $u$ in a right neighborhood of $0$.
Consequently, $f(x+tz+ua) \geq f(x) + tf'(x,z) + uc$ for $u \in [0,t]$,
yielding the desired equality.

We may rewrite the last claim as
$f(x + tz + tua) = f(x) + t f'(x,z) + tuc$
for $t \in [0,1], u \in [0,1]$. From this, we may deduce that
$g(u) = f'(x,z+ua) = f'(x,z) + uc$ for $u \in [0,1]$.
This proves affinity of $g$ near an endpoint, completing the argument.
\end{proof}

We now establish the compact case of Theorem~\ref{T:trans}.
\begin{lemma} \label{L:trans compact}
The conclusion of Theorem~\ref{T:trans} holds if $C$ is compact.
\end{lemma}
\begin{proof}
We may assume that $C$ has nonempty
interior, by replacing $\RR^n$ by a plane containing $C$ of the appropriate
dimension. With this extra hypothesis, we proceed by induction on $n$, with
trivial base case $n=1$.

We have convexity of $f$ by Lemma~\ref{L:trp1 convex}.
It thus suffices to prove that $f$ is transintegral polyhedral
(and hence continuous)
in a neighborhood of any $x \in C$.
By Lemma~\ref{L:dir deriv trp}, the restriction of 
$f'(x,z)$ to any compact TRP subset of $\angle_x C$
is convex transintegral polyhedral in dimension 1.
By applying the induction hypothesis to the intersection of $\angle_x C$
with a rational hyperplane, we may deduce that $f'(x,z)$ is 
continuous, convex, and transintegral polyhedral. By
Theorem~\ref{T:integral-polyhedral}, $f'(x,z)$ is in fact integral
polyhedral.

To prove that $f$ is transintegral polyhedral in a neighborhood of $x$,
it suffices to do so after cutting $C$ into finitely many pieces.
We may thus reduce to the case where $f'(x,z)$ is affine on $\angle_x C$.
Since $\angle_x C$ is a rational polyhedral cone, we may
pick $z_1,\dots,z_l \in \angle_x C \cap \QQ^n$ such that
$\angle_x C$ is the convex hull of the rays from $0$ through $z_1,\dots,z_l$.
We may then rescale $z_1,\dots,z_l$ so that
$f(x + tz_i) = f(x) + tf'(x,z_i)$ for $i=1,\dots,l$ and $t \in [0,1]$.

For any $z$ in the convex hull of $z_1,\dots,z_l$, we now deduce 
(using the affinity of $f'(x,z)$) that $f(x + z) \leq f(x) + f'(x,z)$.
Since $f(x+tz)$ is convex in $t$,
this is only possible if 
$f(x+tz) = f(x) + tf'(x,z)$ for $t \in [0,1]$.
We conclude that $f$ agrees with an integral affine functional
on the convex hull of 
$x, x+z_1,\dots,x+z_l$. As noted above, this completes the proof.
\end{proof}

We now allow allow $C$ which are no longer necessarily bounded.
\begin{defn}
Let $C$ be a TRP subset of $\RR^n$. Define the \emph{small cone} of $C$
at $x$, denoted $\angle'_x C$, to be the set of $z \in \RR^n$ such that
$x + tz \in C$ for \emph{all} $t > 0$; this is again a convex rational polyhedral
cone in $\RR^n$. Moreover, it does not depend on $x$ by the 
following reasoning.
Write $C = \{x \in \RR^n: \lambda_1(x), \dots, \lambda_m(x) \geq 0\}$
for some transintegral affine functionals $\lambda_1,\dots, \lambda_m$.
Write $\lambda_i(x) = \lambda_{i,0}(x) + c_i$ with $\lambda_{i,0}$ linear.
Then $z \in \angle'_x C$ if and only if $x \in C$ and 
$\lambda_{i,0}(z) \geq 0$ for $i=1,\dots,m$. In particular, $\angle'_x C$
does not depend on the choice of $x \in C$; we thus notate it also by 
$\angle' C$.
\end{defn}

\begin{lemma} \label{L:trans closed}
The conclusion of Theorem~\ref{T:trans} holds.
\end{lemma}
\begin{proof}
We may again assume that $C$ has nonempty interior in $\RR^n$;
by slicing $C$ with hyperplanes, we may further assume that 
the small cone $\angle' C$ is strictly convex (i.e., $\angle' C \cap -\angle' 
C = \{0\}$). We now induct on $n$, where we may assume $n \geq 2$
because the case $n=1$ is trivial. 
By the induction hypothesis, the restriction of $f$ to each
boundary facet of $C$ is convex transintegral polyhedral.

As in the proof of Lemma~\ref{L:dir deriv trp}, for each
boundary facet $D$ of $C$, each $i \in \{1,\dots,n\}$, 
and each $a \in \QQ^n$, the function $x \mapsto f'(x, a)$
is constant on the interior of each domain of affinity of the
restriction of $f$ to $D$. In particular, for $x \in D$ outside of a set
of measure zero, $f'(x,a)$ takes only finitely many values.

By Lemma~\ref{L:trans compact}, $f$ is polyhedral on any compact TRP
subset of $C$. In particular, $C$ is covered by
domains of affinity of $f$; to prove that $f$ is polyhedral on all of 
$C$, it suffices to show that $C$ can be covered by finitely many
domains of affinity of $f$ (see Remark~\ref{R:fin-many-affdoms}).
By Lemma~\ref{L:slopes}, it suffices to check that the ambient functionals
on domains of affinity of $f$ can have only finitely many slopes. 

Let $U$ be a domain of affinity of $f$ with ambient functional $\lambda$.
Choose a basis $a_1,\dots,a_n$ of $\QQ^n$ none of whose elements
is contained in $\angle' C \cup (-\angle' C)$ (this is possible because
$\angle' C$ is strictly convex and $n \geq 2$).
For $x \in U$ and $i \in \{1,\dots,n\}$,
the function $f(x+ta_i)$ on $I_{x,a_i}$
is convex transintegral polyhedral, so has a limiting slope at each
endpoint of $I_{x,a_i}$. 
(Note that our hypothesis that $a_i \notin \angle' C \cup (-\angle' C)$
ensures that $I_{x,a_i}$ is compact.)
By the previous paragraph, for $x$ away from a set
of measure zero, these limiting slopes are themselves confined to a finite
set. Since $f$ is convex, the slope of $f(x+ta_i)$ at $t=0$ is now also
constrained to a finite set. This conclusion for $i=1,\dots,n$
constrains the slope of $\lambda$ to a finite set, proving the claim.
\end{proof}

\subsection{Variation of subsidiary radii}
\label{S:multi-var}

In this subsection, we will extend Theorem~\ref{T:1-dim-var} into a
higher-dimensional generalization (Theorem~\ref{T:main-thm}).  We keep Hypothesis~\ref{H:K} and Notation~\ref{N:K-m-J}.
We begin by introducing the setup of \cite[Section~4.1]{kedlaya-part3}.

\begin{notation}
Throughout this subsection, we put $I = \{\serie{}n\}$ for notational simplicity.
\end{notation}

\begin{notation}
For $X$ an $n$-tuple:
\begin{itemize}
\item
for $A$ an $n \times n$ matrix, write $X^A$ for the $n$-tuple
whose $j$-th entry is $\prod_{i=1}^n x_i^{A_{ij}}$;
\item
for $c$ a number, put $X^c = (x_1^c, \dots, x_n^c)$.
\end{itemize}
\end{notation}

\begin{defn}
For a subset $C \subset \RR^n$, let $e^{-C}$ denote the subset $\{e^{-r_I}: r_I \in C\} \subset (0, +\infty)^n$.
A subset $S$ of $[0, +\infty)^n$ is \emph{log-(T)RP} if $S$ is the closure of
$\stackrel\circ{S} = e^{-C}$ for some (T)RP subset $C$ of $\RR^n$.  We say $S$ is \emph{ind-log-(T)RP} if it is a union of an increasing sequence of log-(T)RP sets $S_\alpha$; we denote $\stackrel\circ{S} = \cup_\alpha \stackrel\circ{S}_\alpha$.  For instance, any open subset of $[0, +\infty)^n$ is covered by ind-log-RP subsets.
\end{defn}

\begin{caution}
The subset $(0, 1]$ is an ind-log-TRP subset but not a log-TRP subset.
By contrast, $[0,1]$ is a log-TRP subset.
\end{caution}

\begin{defn}
Let $C \subset \RR^n$ be a TRP subset defined by \eqref{E:TRP-set}, where $\lambda_s(x_I) = a_{s, 1}x_1 + \cdots + a_{s, n}x_n + b_s$ for $a_{s, i} \in \ZZ$ and $s= \serie{}r$.  Denote the closure of $e^{-C}$ in $[0, +\infty)^n$ by $S$.  Define $A_K(S)$ to be the 
subspace of the (Berkovich) analytic $n$-space with coordinates $t_1, 
\dots, t_n$ satisfying the condition $(|t_1|, \dots, |t_n|) \in S$.  Precisely, 
\[
\Gamma(A_K(S), \calO) = K \langle t_I^{a_{1, I}} / e^{-b_1}, \dots, 
t_I^{a_{r, I}} / e^{-b_r} \rangle.
\]
For an ind-log-TRP subset $S = \cup_\alpha S_\alpha$, we define $A_K(S) = \cap_\alpha A_K(S_\alpha)$.
\end{defn}

\begin{defn} \label{D:total diff mod}
Let $S$ be an ind-log-TRP subset of $[0, +\infty)^n$.  A ($\partial_{I\cup J}$-)\emph{differential module} $M$ over $X = A_K(S)$ is a locally free coherent sheaf together with an integrable connection
$$
\nabla: M \rar M \otimes \left( \bigoplus_{j=1}^m \calO_X \cdot du_j \oplus \bigoplus_{i=1}^n \calO_X \cdot dt_i \right).
$$
We label the derivations $\serie{\partial_}m$ as usual,
and put
$\partial_{m+1} = \partial_{t_1}, \dots, \partial_{m+n} = \partial_{t_n}$.
\end{defn}

\begin{notation}
For $\eta_I = (\serie{\eta_}n) \in \stackrel\circ{S}$, let $F_{\eta_I}$ be the completion of 
$K(t_I)$ with respect to the $\eta_I$-Gauss norm.  Write $f_l(M, r_I) = 
-\log IR(M \otimes F_{e^{-r_I}}; l)$ and 
$F_l(M, r_I) = f_1(M, r_I) + \cdots + 
f_l(M, r_I)$ for $l = 1, \dots, \rank M$.
\end{notation}

\begin{lemma}\label{L:toroidal-transform}
Given $\eta_I \in (0,+\infty)^n$ and $A \in \GL_n (\ZZ)$, 
let $M$ be a differential module over $F_{\eta_I^A}$, 
and let $h_A^*: F_{\eta_I^A} \rar F_{\eta_I}$ be given by 
$t_I \mapsto t_I^A$.  Then $IR(M) = IR(h_A^*M)$.
\end{lemma}
\begin{proof}
This follows from \cite[Proposition~4.2.7]{kedlaya-part3}
(which is itself an immediate consequence of 
\cite[Lemma~4.1.5]{kedlaya-part3}) applied to $A$ and $A^{-1}$.
\end{proof}

\begin{theorem}\label{T:main-thm}
Let $S$ be an ind-log-TRP subset of $[0, +\infty)^n$,
and let $M$ a differential module of rank $d$ over $A_K(S)$.
\begin{enumerate}
\item[(a)]
(Continuity)
For $l=\serie{}d$, the functions $f_l(M, r_I)$ and 
$F_l(M, r_I)$ are continuous.
\item[(b)]
(Convexity)
For $l=1, \dots, d$, the function $F_l(M, r_I)$ is convex.
\item[(c)]
(Polyhedrality)
For $r_I \in -\log \stackrel\circ{S}$,
if $l=d$ or $f_l(M, r_I) > f_{l+1}(M, r_I)$, then 
$F_l(M, r_I)$ is transintegral polyhedral in some neighborhood of $r_I$. 
Moreover, 
on any TRP subset of $-\log \stackrel\circ{S}$,
$d! F_l(M, r_I)$ and $F_d(M, r_I)$ are transintegral polyhedral functions.

\item[(d)] (Monotonicity)
Assume that $S$ is log-TRP.  Then for any $r_I, r'_I \in -\log \stackrel\circ{S}$, if $r_i \leq r'_i$ for $i \in I$ and $(1-t)r_I + tr'_I \in -\log \stackrel\circ{S}$ for any $t \in [0, +\infty)$, then $F_l(M, r_I) \geq F_l(M, r'_I)$ for $l = \serie{}d$.
\end{enumerate}
\end{theorem}
\begin{proof}
We first prove (a)-(c).  We need only verify that, for $l = \serie{}d$, 
$d! F_l(M, r_I)$ and $F_d(M,r_I)$
satisfy the conditions of Theorem~\ref{T:trans}.  
Moreover, by translating and enlarging $K$ if necessary, it suffices to 
check the hypothesis of Theorem~\ref{T:trans} 
for $I_{x,a}$ in the case $x=0$.

It suffices to consider $a = a_I \in \ZZ^n$ with $\gcd(a_I) = 1$.  
Let us describe $f_l(M, a_It)$ and $F_l(M, a_It)$ for $l = \serie{}d$ and 
$t \in I_{0, a_I}$.  
Pick an $n \times n$ invertible integral matrix $A$ with $(a_I)$ as the 
first row. 
Equip  $A_K(S^{A^{-1}})$ with the coordinates $(s_I)$,
and define the toroidal transform $\phi: A_K(S^{A^{-1}}) \rar A_K(S)$
by $\phi^*(t_I) = s_I^A$, where $S^{A^{-1}} = \{X^{A^{-1}}|X \in S\}$.
By Lemma~\ref{L:toroidal-transform}, $f_l(M, a_It) = 
f_l(\phi^*M, (a_IA^{-1})t)$.
The theorem follows from Theorem~\ref{T:1-dim-var}.

To prove (d), by continuity, we may assume that $r_I - r'_I$ are all rational numbers.  By an argument as in the previous paragraph,
we may reduce to the one-dimensional case.
In this case, we get a differential module over a disc,
so the desired statement follows from Theorem~\ref{T:1-dim-var}(c).
\end{proof}

\subsection{Decomposition by subsidiary radii}

To conclude, we extend the theorems of \S\ref{S:decomp-mult-1-dim}
to higher-dimensional spaces.

\begin{lemma}\label{L:decomp-simplicial}
Suppose $r \in \{0, \dots, n\}$.  
Put $C = \{(x_I)| x_I \geq 0, x_1 + \cdots + x_r \leq 1\} \subset \RR^n$,
and let $C_\epsilon$ be any TRP subset of $\RR^n$ containing $C$ in
its interior.
Let $S$ (resp. $S_\epsilon$) denote the closure of $e^{-C}$ (resp. $e^{-C_\epsilon}$) in $[0, +\infty)^n$, which is a log-TRP subset.  
Let $M$ be a differential module of rank $d$ over $A_K(S_\epsilon)$.  Suppose that the following conditions hold for some $l \in \{\serie{}d-1\}$.
\begin{itemize}
\item [(a)] The function $F_l(M, r_I)$ is affine for $(r_I) \in C_\epsilon$.
\item [(b)] We have $f_l(M, r_I) > f_{l+1}(M, r_I)$ for $(r_I) \in C_\epsilon$.
\end{itemize}
Then $M$ admits a unique direct sum 
decomposition over $A_K(S)$ separating the first $l$ subsidiary radii of $M \otimes F_{e^{-r_I}}$ for any $(r_I) \in C$.
\end{lemma}
\begin{proof}
Note that $\Gamma(A_K(S), \calO) = 
K \langle t_I, e^{-1} /t_1\cdots t_r \rangle$ 
may be embedded into the completion $F_{1,\dots,1}$ of $K(t_1,\dots,t_n)$
for the $(1,\dots,1)$-Gauss norm.
For $i=1,\dots,n$, let $F^{(i)}_{1,\dots,1}$ 
be the completion of $K(t_1,\dots,\widehat{t_i},\dots,t_n)$ for the
$(1,\dots,1)$-Gauss norm; then the image of
$\Gamma(A_K(S), \calO)$ also belongs to each of the subrings
\[
F^{(i)}_{1, \dots, 1} \langle e^{-1} /t_i, t_i \rangle \quad (i = 1, \dots, r); 
\qquad
F^{(i)}_{1, \dots, 1} \langle t_i \rangle \quad (i = r+1, \dots, n),
\]
In fact, it is equal to the intersection of these subrings; this is true
because $C$ is the convex hull of the union of the segments
\begin{gather*}
\{(x_1,\dots,x_n): 0 \leq x_i \leq 1; \quad x_j = 0 \quad (j \neq i)\} \qquad
(i=1,\dots,r) \\
\{(x_1,\dots,x_n): 0 \leq x_i; \quad x_j = 0 \quad (j \neq i)\} \qquad
(i=r+1,\dots,n).
\end{gather*}
Consequently, by Lemma~\ref{L:proj-intersect}, it suffices 
to prove the decomposition over the
rings $F^{(i)}_{1, \dots, 1} \langle e^{-1}/t_i, t_i 
\rangle$ for $i = 1, \dots, r$
and $F^{(i)}_{1, \dots, 1} \langle t_i \rangle$ for $i = r+1, \dots, n$.
The former case follows by applying Theorem~\ref{T:decomp-multi-1-dim-annulus} to $M \otimes F_{1, \dots, 1} \langle e^{-1-\epsilon} / t_i, t_i / e^\epsilon \rangle$ for $i = 1, \dots, r$ for some $\epsilon > 0$; the latter case follows by applying Theorem~\ref{T:decomp-multi-1-dim-disc} to $F_{1, \dots, 1} \langle t_i / e^\epsilon \rangle$ for $i = r+1, \dots, n$ for some $\epsilon > 0$.
\end{proof}

\begin{theorem}\label{T:decomp-full}
Let $S$ be a ind-log-TRP subset of $[0, +\infty)^n$, and let $M$ a differential module of rank $d$ over $A_K(\inte(S))$.  Suppose that the following conditions hold for some $l \in \{\serie{}d-1\}$.
\begin{itemize}
\item [(a)] The function $F_l(M, r_I)$ is affine for $(r_I) \in \inte( -\log \stackrel\circ{S})$.
\item [(b)] We have $f_l(M, r_I) > f_{l+1}(M, r_I)$ for $(r_I) \in \inte(-\log \stackrel\circ{S})$.
\end{itemize}
Then $M$ admits a unique direct sum 
decomposition over $A_K(\inte(S))$ separating the first $l$ subsidiary radii of $M \otimes F_{e^{-r_I}}$ for any $(r_I) \in \inte(-\log \stackrel\circ{S})$.
\end{theorem}
\begin{proof}
We can cover $\inte(S)$ by log-TRP subsets $S_\alpha \subset \inte(S)$ 
such that for each point of $x \in \inte(S)$,
there exists a neighborhood of $x$ contained in some $S_\alpha$.  
Moreover, we can choose those $S_\alpha$ to be simplicial, i.e., under a toroidal transform and rescaling, each $S_\alpha$ can be transformed
into the form desired for Lemma~\ref{L:decomp-simplicial}.  Since $S_\alpha$ lies in the interior of $S$, the decomposition follows from Lemma~\ref{L:decomp-simplicial} by gluing 
the decompositions obtained on each of the $S_\alpha$.
\end{proof}

\begin{lemma} \label{L:decomp-partial}
Suppose $r \in \{0, \dots, n\}$.  
Put $C = \{(x_I)| x_I \geq 0, x_1 + \cdots + x_r < 1\} \subset \RR^n$,
and let $C_\epsilon$ be any TRP subset of $\RR^n$ containing $C$ in
its interior.
Let $S_\epsilon$ denote the closure of $e^{-C_\epsilon}$ in $[0, +\infty)^n$, 
which is a log-TRP subset.  
Let $S$ be the set of points $(s_I) \in S_\epsilon$ such that
$s_I \leq 1$ and $s_1 \cdots s_r > e^{-1}$.
Let $R$ be the subring of $\Gamma(A_K(S_\epsilon), \calO)$ consisting of 
those $f$ for which
$|f|_{s_I}$ is bounded over $(s_I) \in S$.
Let $M$ be a differential module of rank $d$ over $A_K(S_\epsilon)$.  
Suppose that the following conditions hold for some $l \in \{\serie{}d-1\}$.
\begin{itemize}
\item [(a)] The function $F_l(M, r_I)$ is affine for $(r_I) \in C_\epsilon$.
\item [(b)] We have $f_l(M, r_I) > f_{l+1}(M, r_I)$ for $(r_I) \in C_\epsilon$.
\end{itemize}
Then $M \otimes R$ admits a unique direct sum 
decomposition separating the first $l$ subsidiary radii of $M \otimes F_{e^{-r_I}}$ for any $(r_I) \in C$.
\end{lemma}
\begin{proof}
Let $F$ be the completion of $\Frac R$ for the 
$(1,\dots,1)$-Gauss norm.
Define $F^{(i)}_{1,\dots,1}$ as in the proof of 
Lemma~\ref{L:decomp-simplicial}. Then inside $F$,
$R$ is the intersection of the rings
\[
F^{(i)}_{1,\dots,1} \langle 1/t_i^{-1}, t_i^{-1}/e \rrbracket_0
\qquad (i=1,\dots,r); \qquad
F^{(i)}_{1, \dots, 1} \langle t_i \rangle \quad (i = r+1, \dots, n).
\]
We may thus argue as in Lemma~\ref{L:decomp-simplicial}, but using
Theorem~\ref{T:decomp-multi-1-dim-annulus-bounded} instead
of Theorems~\ref{T:decomp-multi-1-dim-annulus}
and~\ref{T:decomp-multi-1-dim-disc}.
\end{proof}

\begin{theorem} \label{T:decomp-partial}
Let $S$ be a log-TRP subset of $[0, +\infty)^n$.
Let $R$ be the subring of $\Gamma(A_K(\inte(S)), \calO)$ consisting of those $f$ for which
$|f|_{s_I}$ is bounded over $s_I \in \inte(S)$.
Let $M$ be a differential module of rank $d$ over $A_K(S)$.
Suppose that the following conditions hold for some $l \in \{\serie{}d-1\}$.
\begin{itemize}
\item [(a)] The function $F_l(M, r_I)$ is affine for $(r_I) \in -\log \stackrel\circ{S}$.
\item [(b)] We have $f_l(M, r_I) > f_{l+1}(M, r_I)$ for $(r_I) \in -\log \stackrel\circ{S}$.
\end{itemize}
Then $M \otimes R$ admits a unique direct sum 
decomposition separating the first $l$ subsidiary radii of $M \otimes F_{e^{-r_I}}$ for any $(r_I) \in \inte(-\log \stackrel\circ{S})$.
\end{theorem}
\begin{proof}
Analogous to Theorem~\ref{T:decomp-full}, except using
Lemma~\ref{L:decomp-partial} instead of Lemma~\ref{L:decomp-simplicial}.
\end{proof}

\begin{remark}
It may be helpful to illustrate the argument needed to reduce
Theorem~\ref{T:decomp-partial} to Lemma~\ref{L:decomp-partial} with
an explicit example. Take $S = [0, 1]^2$, so that
$R = \gotho_K \llbracket x,y \rrbracket \otimes_{\gotho_K} K$.
We must partition $\inte(-\log \stackrel{\circ}{S}) = (0, +\infty)^2$
into regions to which Lemma~\ref{L:decomp-partial} may be applied.
One such partition consists of
\begin{gather*}
\{(x,y) \in \RR^2: 0 < x, 0 < y \leq \min\{x, 1\}\}, \\
\{(x,y) \in \RR^2: 0 < y, 0 < x \leq \min\{y, 1\}\}, \\
\{(x,y) \in \RR^2: 1 \leq x, 1 \leq y \}.
\end{gather*}
Since the parts all contain $(1,1)$, we can 
glue the three resulting decompositions together by matching them
on $M \otimes F_{e^{-1},e^{-1}}$.
\end{remark}

\begin{remark}
Note that Lemma~\ref{L:decomp-partial} is not a special case
of Theorem~\ref{T:decomp-partial}. We leave the formulation and proof
of a common generalization as 
a somewhat awkward exercise for the reader.
\end{remark}

\begin{remark} \label{R:bounded-in-discrete-case-mult}
By Remark~\ref{R:bounded-in-discrete-case-mult1},
in Theorem~\ref{T:decomp-partial},
if $\log |K^\times| \subseteq \QQ$ and $-\log 
\stackrel{\circ}{S}$ is RP,
we may also take $M$ to be defined over
$R$. For example, if $K$ carries the trivial valuation (forcing $p=0$)
and 
\[
S = \{ (x,y) \in (0, 1]^2: xy = e^{-1} \},
\]
then $R = K \llbracket x,y \rrbracket [x^{-1}, y^{-1}]$. This
example can be used in the study of good formal structures for
flat holomorphic connections; 
however, one needs to refine Theorem~\ref{T:decomp-partial} slightly
in case $p=0$,
to remove the need for strict
inequality on the boundary of $-\log \stackrel{\circ}{S}$.
For this, we defer to \cite{kedlaya-goodformal1}.
\end{remark}


\begin{thebibliography}{99}

\bibitem{baldassarri-divizio}
F. Baldassarri and L. di Vizio, Continuity of the radius of convergence
of $p$-adic differential equations on Berkovich analytic spaces, 
arXiv preprint 0709.2008v3 (2008).

\bibitem{berkovich}
V.G. Berkovich,
\textit{Spectral Theory and Analytic Geometry over Non-Archimedean Fields},
Math. Surveys and Monographs 33, Amer. Math. Soc., Providence, 1990.

\bibitem{christol-dwork}
G. Christol and B. Dwork, Modules diff\'erentielles sur les
couronnes, \textit{Ann. Inst. Fourier} \textbf{44} (1994), 663--701.

\bibitem{dgs}
B. Dwork, G. Gerotto, and F. Sullivan,
\textit{An Introduction to $G$-Functions},
Annals of Math. Studies 133, Princeton University Press, Princeton, 1994.

\bibitem{eisenbud}
D. Eisenbud, \textit{Commutative Algebra}, Graduate Texts in Math. 150,
Springer-Verlag, New York, 1995.

\bibitem{ega41}
A. Grothendieck, 
\'{E}l\'ements de g\'eom\'etrie alg\'ebrique. {IV}. 
\'{E}tude locale des sch\'emas et des morphismes de sch\'emas. {I},
\textit{Publ. Math. IH\'ES} \textbf{20} (1964).

\bibitem{kedlaya-part1}
K.S. Kedlaya, Semistable reduction for overconvergent $F$-isocrystals, I:
    Unipotence and logarithmic extensions,
\textit{Compos. Math.} \textbf{143} (2007), 1164--1212.

\bibitem{kedlaya-swan1}
K.S. Kedlaya, Swan conductors for $p$-adic differential modules, I:
A local construction, 
\textit{Alg. and Num. Theory} \textbf{1} (2007), 269--300.

\bibitem{kedlaya-fake}
K.S. Kedlaya, The $p$-adic local monodromy theorem for fake annuli, 
\textit{Rend. Sem. Math. Padova} \textbf{118} (2007), 101--146.

\bibitem{kedlaya-part3}
K.S. Kedlaya, Semistable reduction for overconvergent $F$-isocrystals, III:
Local semistable reduction at monomial valuations, arXiv preprint 
math/0609645v3 (2008); to appear in \textit{Compos. Math.}

\bibitem{kedlaya-course}
K.S. Kedlaya, $p$-adic differential equations (version of 15 Dec 08),
course notes available at \texttt{http://math.mit.edu/\~{}kedlaya/papers/}.

\bibitem{kedlaya-swan2}
K.S. Kedlaya, Swan conductors for $p$-adic differential modules,
II: Global variation, arXiv preprint 0705.0031v2 (2008).

\bibitem{kedlaya-goodformal1}
K.S. Kedlaya, Good formal structures for flat meromorphic
connections, I: Surfaces, 
arXiv preprint 0811.0190v2 (2008).

\bibitem{kkms}
G. Kempf, F.F. Knudsen, D. Mumford, and B. Saint-Donat, 
\textit{Toroidal embeddings. I}, Lecture Notes in Math. \textbf{339},
Springer-Verlag, Berlin-New York, 1973.

\bibitem{matsuda-dwork}
S. Matsuda,
Conjecture on Abbes-Saito filtration and Christol-Mebkhout filtration,
in \textit{Geometric aspects of Dwork theory}, Vol. II, 
de Gruyter, Berlin, 2004, 845--856.

\bibitem{ore}
O. Ore, Theory of non-commutative polynomials,
\textit{Annals of Math.} \textbf{34} (1933), 480--508.

\bibitem{ribenboim}
P. Ribenboim, \textit{The Theory of Classical Valuations},
Springer-Verlag, New York, 1999.

\bibitem{sabbah}
C. Sabbah,
\'Equations diff\'erentielles \`a points singuliers irr\'eguliers 
et ph\'enom\`ene de Stokes en dimension 2,
\textit{Ast\'erisque} \textbf{263} (2000).

\bibitem{schneider}
P. Schneider, \textit{Nonarchimedean Functional Analysis},
Springer-Verlag, Berlin, 2002.

\bibitem{young}
P.T. Young, Radii of convergence and index for $p$-adic differential
operators, \textit{Trans. Amer. Math. Soc.} \textbf{333} (1992),
769--785.

\end{thebibliography}
\end{document}